\pdfoutput=1

\documentclass{amsart}

\usepackage[utf8]{inputenc}
\usepackage[T1]{fontenc}
\usepackage[english]{babel}
\usepackage{ifthen}
\usepackage{enumitem}
\usepackage{txfonts}
\usepackage{dsfont}
\usepackage{mathrsfs}
\usepackage{bm}
\usepackage{amssymb}
\usepackage{bbm}

\usepackage{mathtools}
\usepackage{esint}

\usepackage[babel]{csquotes}
\usepackage[protrusion=true,expansion=true,babel=true,final]{microtype}
\usepackage[unicode,bookmarksopen]{hyperref}

\numberwithin{equation}{section}

\newtheorem{lemma}{Lemma}[section]
\newtheorem{corollary}[lemma]{Corollary}
\newtheorem{proposition}[lemma]{Proposition}
\newtheorem{theorem}[lemma]{Theorem}

\theoremstyle{definition}
\newtheorem{definition}[lemma]{Definition}

\newtheorem{remark}[lemma]{Remark}

\newlist{thm_enum}{enumerate}{1}
\setlist[thm_enum]{label=\normalfont(\alph*)}
\newlist{def_enum}{enumerate}{1}
\setlist[def_enum]{label=\normalfont(\roman*)}
\newlist{equiv_enum}{enumerate}{1}
\setlist[equiv_enum]{label=\normalfont(\roman*)}

\newcommand{\IZ}{\mathbb{Z}}

\newcommand{\IN}{\mathbb{N}}
\newcommand{\IR}{\mathbb{R}}
\newcommand{\IC}{\mathbb{C}}

\newcommand{\IP}{\mathbb{P}}

\newcommand{\abs}[1]{\left\lvert#1\right\rvert}
\newcommand{\normalabs}[1]{\lvert#1\rvert}

\newcommand{\norm}[1]{\left\lVert#1\right\rVert}
\newcommand{\normalnorm}[1]{\lVert#1\rVert}
\newcommand{\biggnorm}[1]{\biggl\lVert#1\biggr\rVert}

\newcommand{\R}[2][\empty]{
	\ifthenelse{\equal{#1}{\empty}}
		{\mathcal{R}\left\{#2\right\}}
		{\mathcal{R}_{#1}\left\{#2\right\}}
}
\newcommand{\ip}[2]{\langle #1, #2 \rangle}
\newcommand{\ud}{\, \mathrm{d}}

\makeatletter
\newcommand{\LeftEqNo}{\let\veqno\@@leqno}
\makeatother

\renewcommand{\d}{\mathop{}\!d}

\renewcommand{\epsilon}{\varepsilon}
\renewcommand{\phi}{\varphi}
\renewcommand{\vec}[1]{\bm{#1}}

\DeclareMathOperator{\sgn}{sgn}

\DeclareMathOperator{\Id}{Id}

\DeclareMathOperator{\supp}{supp}

\DeclareMathOperator*{\esssup}{ess\,sup}
\DeclareMathOperator{\ran}{ran}

\DeclareMathAlphabet{\mathpzc}{OT1}{pzc}{m}{it}

\allowdisplaybreaks[4]

\begin{document}

\title{Weighted Estimates for Operator-Valued Fourier Multipliers}

\begin{abstract}
	We establish Littlewood--Paley decompositions for Muckenhoupt weights in the setting of UMD spaces. As a consequence we obtain two-weight variants of the Mikhlin multiplier theorem for operator-valued multipliers. We also show two-weight estimates for multipliers satisfying Hörmander type conditions.
\end{abstract}

\author{Stephan Fackler}
\address{Institute of Applied Analysis, Ulm University, Helmholtzstr.\ 18, 89069 Ulm, Germany}
\email{stephan.fackler@alumni.uni-ulm.de}

\author{Tuomas P.\ Hytönen}
\address{Department of Mathematics and Statistics, P.O.B.\ 68 (Pietari Kalmin katu 5), FI-00014 University of Helsinki, Finland}
\email{tuomas.hytonen@helsinki.fi}

\author{Nick Lindemulder}
\address{Delft Institute of Applied Mathematics, Delft University of Technology, P.O. Box 5031, 2600 GA Delft, The Netherlands}
\email{N.Lindemulder@tudelft.nl}

\thanks{The first author was supported by the DFG grant AR 134/4-1 ``Regularität evolutionärer Probleme mittels Harmonischer Analyse und Operatortheorie''. Parts of the work were done during a research stay and a visit of the first author at the University of Helsinki and the Delft University of Technology. He thanks both institutes for their hospitality.
The third author is supported by the VIDI subsidy 639.032.427 of the Netherlands Organisation for Scientific Research (NWO). Parts of the work were done during a visit of the third author to the University of Helsinki. He thanks the institute for its hospitality.}

\maketitle

\section{Introduction}
	
	The classical multiplier theorems by Marcinkiewicz, Hörmander and Mikhlin have been extended in at least two different fundamental directions. The first one replaces the Lebesgue measure by a class of weights on $\mathbb{R}^n$. For $p \in (1, \infty)$ we say that a locally integrable function $\omega\colon \IR^n \to \IR_{\ge 0}$ belongs to the \emph{Muckenhoupt class} $\mathcal{A}_p(\IR^n; \mathcal{C}_n)$ if
	\begin{equation*}
		[\omega]_{\mathcal{A}_p(\IR^n;\mathcal{C}_n)} \coloneqq \sup_{A \in \mathcal{C}_n} \left( \frac{1}{\abs{A}} \int_{A} \omega(x) \d x \right) \left( \frac{1}{\abs{A}} \int_{A} \omega(x)^{-\frac{1}{p-1}} \d x \right)^{p-1} < \infty.
	\end{equation*}
	Here $\mathcal{C}_n$ is either the collection $\mathcal{Q}_n$ or $\mathcal{R}_n$ of all cubes or rectangles in $\IR^n$ of positive and finite measure whose sides are parallel to the coordinate axes. Clearly, in the one dimensional case one has $\mathcal{A}_p(\IR; \mathcal{Q}_1) = \mathcal{A}_p(\IR; \mathcal{R}_1)$, whereas in the higher dimensional case the strict inclusion $\mathcal{A}_p(\IR^n; \mathcal{R}_n) \subsetneq \mathcal{A}_p(\IR^n; \mathcal{Q}_n)$ holds. The class $\mathcal{A}_p(\IR^n; \mathcal{Q}_n)$ goes back to~\cite{Muck72}, where it was used to give a real variable characterization of the weights $\omega$ for which the Hardy--Littlewood maximal operator remains bounded on the weighted space $L^p_{\omega}(\IR^n)$. Afterwards the classical Fourier multiplier theorems were extended by Kurtz~\cite{Kur80} and Kurtz/Wheeden~\cite{KurWhe79} to weights in $\mathcal{A}_p(\IR^n;\mathcal{R}_n)$. The crucial tool in the proofs is some variant of the Littlewood--Paley decomposition. Recently, sharp weighted estimates for the Littlewood-Paley square function and Marcinkiewicz multipliers were considered in \cite{Ler18} in the one-dimensional case.
Another recent work on weighted estimates for Fourier multipliers is \cite{Amenta&Lorist&Veraar2017_Fourier}, where various extensions of the Coifman–-Rubio de Francia–-
Semmes multiplier theorem to operator-valued multipliers on Banach function
spaces were obtained.

	Motivated from applications in non-linear partial differential equations, there has been much interest in Fourier multipliers on vector-valued $L^p$-spaces. It is understood that a reasonable Fourier multiplier theory is only possible on a certain class of complex Banach spaces, the class of UMD spaces. For such spaces a vector-valued analogue of the Littlewood--Paley decomposition was obtained by Bourgain~\cite{Bou86}. Based on this decomposition variants of the Mikhlin multiplier theorem for scalar multipliers were obtained by Zimmermann in~\cite{Zim89}.
	
	The case of operator-valued multipliers involves a concept called $\mathcal{R}$-boundedness (see Section~\ref{sec:r-boundedness} for details) and a corresponding Mikhlin multiplier theorem has been established in~\cite{Wei} in the one and in~\cite{StrWei07} in the higher dimensional case. Our goal is to unify these generalizations and to show multiplier theorems for operator-valued multipliers in the $\mathcal{A}_p$-setting. Here the crucial step is to generalize Bourgain's Littlewood--Paley decomposition to $\mathcal{A}_p$-weights.
	
	Further, we go beyond the one-weight setting and generalize the Mikhlin multiplier theorem to the two-weight setting, i.e.\ we consider the boundedness of multipliers between $L^p_{\sigma}(\IR^n;X)$ and $L^p_{\omega}(\IR^n;Y)$. Here one replaces the $\mathcal{A}_p$-condition with its two weight analogue
	\begin{equation*}
		[\omega,\sigma]_{\mathcal{A}_p(\IR^n; \mathcal{C}_n)} \coloneqq \sup_{A \in \mathcal{C}_n} \left( \frac{1}{\abs{A}} \int_A \omega(x) \d x \right) \left( \frac{1}{\abs{A}} \int_A \sigma(x)^{-\frac{1}{p-1}} \d x \right)^{p-1}.
	\end{equation*}
	Our two-weight multiplier results seem to be new even in the case of scalar-valued multipliers.

By its very own nature, the approach based on Littlewood--Paley theory only yields results for $\mathcal{A}_p$-weights with respect to rectangles. However, we also give multiplier results for weights with respect to cubes only. In contrast to the previous results the made assumptions are of Hörmander instead of Mikhlin type. In this way we obtain weighted multiplier results which generalize~\cite{KurWhe79} in three directions: first we deal with operator-valued multipliers, secondly we work in a two-weight setting and thirdly we obtain estimates with explicit dependencies on the weight characteristics. As an application we use our established multiplier theorems to recover directly some extrapolation results for maximal $L^p$-regularity of evolution equations.

In order to give an impression of the paper, let us now state three Fourier multiplier results we are able to prove.
For simplicity we restrict ourselves here to the one-weight setting. Besides this restriction, the three results below are special cases or simplified versions of more general and/or technical results from the main part of the paper.
	
The first result follows the Littlewood--Paley approach (Section~\ref{sec:LP}; Section~\ref{sec:LP1d} in the one-dimensional case) and is therefore restricted to the setting of rectangular $\mathcal{A}_p$-weights.
\begin{theorem}\label{thm:mikhlin_weighted;intro}
		Let $X$ be a UMD space, $p \in (1,\infty)$, and $\omega \in \mathcal{A}_p(\IR^n; \mathcal{R}_n)$.
Let $m \in L^{\infty}(\IR^{n};\mathcal{B}(X))$ be such that $\partial^{\alpha}m$ is continuous on $\IR^{n}_{*}=[\IR \setminus\{0\}]^{n}$ for each $|\alpha|_{\infty} \leq 1$.
Then
    			\begin{equation*}
					\norm{T_m}_{\mathcal{B}(L^{p}_{\omega}(\IR^{n};X))} \lesssim_{X,p,n,\omega} \sup_{|\alpha|_{\infty} \le 1} \mathcal{R}\left\{\abs{\xi}^{\abs{\alpha}} \partial^{\alpha} m(\xi): \xi \in \IR^{n}_{*}\right\}.
    			\end{equation*}
\end{theorem}

The above theorem is a special case of Theorem~\ref{thm:mikhlin_weighted}.\ref{item:thm:mikhlin_weighted;UMD}. Part \ref{item:thm:mikhlin_weighted;UMDalpha} of that theorem is concerned with the case that $X$ in addition satisfies Pisier's property $(\alpha)$ and provides $\mathcal{R}$-boundedness of a set of Fourier multipliers. We furthermore obtain a version for anisotropic symbols on mixed-norm spaces (see Theorem~\ref{thm:mikhlin_weighted_mixed-norm_anisotropic}),
extending \cite[Theorem~3.2]{Hyt07b} (see also~\cite[Section~7]{Hyt07c}) to the weighted setting.

As an application we use Theorem~\ref{thm:mikhlin_weighted;intro} to recover directly some extrapolation results for maximal $L^p$-regularity of evolution equations (Section~\ref{sec:application_max-Lp-reg}).

We also give multiplier results for $\mathcal{A}_{p}$-weights with respect to cubes (Section~\ref{sec:cubular}), where we pass from the multiplier perspective to the perspective of singular integral operators.

\begin{theorem}\label{thm:mikhlin_weighted_non-sharp_dep_weight;intro}
		Let $X$ be a UMD space, $p \in (1,\infty)$, and $\omega \in \mathcal{A}_p(\IR^n; \mathcal{Q}_n)$.
Let $m \in L^{\infty}(\IR^{n};\mathcal{B}(X))$ be such that $\partial^{\alpha}m$ is continuous on $\IR^{n} \setminus \{0\}$ for each $|\alpha|_{1} \leq n$.
Then
    			\begin{equation*}
					\norm{T_m}_{\mathcal{B}(L^{p}_{\omega}(\IR^{n};X))} \lesssim_{X,p,n,\omega} \sup_{|\alpha|_{1} \le n} \mathcal{R}\left\{\abs{\xi}^{\abs{\alpha}} \partial^{\alpha} m(\xi): \xi \in \IR^{n}\setminus\{0\}\right\}.
    			\end{equation*}
\end{theorem}

The above stated theorem actually is a consequence of Corollary~\ref{cor:estimates_mutliplier;GirWei03} and the fact that $\mathcal{A}_{p} = \bigcup_{q \in (1,p)}\mathcal{A}_{q}$ (see Remark~\ref{rmk:cor:estimates_mutliplier;GirWei03}).
The assumption on the weight $\omega$ in Corollary~\ref{cor:estimates_mutliplier;GirWei03} is $\omega \in \mathcal{A}_{p/r}$ or $\omega^{-\frac{1}{p-1}} \in \mathcal{A}_{p'/r}$ for suitable $r \in (1,p)$, with estimates explicitly depending on the weight characteristics.

Moreover, Corollary~\ref{cor:estimates_mutliplier;GirWei03} is in turn a consequence of our estimates for Fourier symbols satisfying conditions of H\"ormander type instead of stronger conditions of Mikhlin type.
These estimates take a geometric property of the Banach space $X$, namely Fourier type, into account, linking the required smoothness of the symbol $m$ to the geometry of $X$. Passing to the stronger H\"ormander conditions this gives a weighted extension of \cite[Corollary~4.4]{GirWei03} (see Corollary~\ref{cor:estimates_mutliplier;GirWei03}).

\section{Preliminaries}

\subsection{The Basic Setting}\label{sec:r-boundedness}	
	
	 We now give exact definitions and fix the setting. A general reference, in which more details on these topics can be found, is \cite{HNVW16}. In the following let $X,Y$ be Banach spaces which are always assumed to be complex. We denote by $\mathcal{S}(\IR^n;X)$ the space of all $X$-valued Schwartz functions. Further let $\mathcal{S}'(\IR^n;X)$ be the associated space of distributions, i.e.\ the space of all continuous linear mappings $\phi\colon \mathcal{S}(\IR^n) \to X$. %
	 For weights $\omega$ and $\sigma$, i.e.\ measurable functions $\IR^n \to [0,\infty]$ that take their values in $(0,\infty)$ almost everywhere, let
	 \begin{align*}
	 	L^p_{\omega}(\IR^n;X) \coloneqq \left\{ f \colon \IR^n \to X \text{ Bochner measurable}: \int_{\IR^n} \norm{f(x)}_X^p \omega(x) \d x < \infty \right\}.
	 \end{align*}
	 Here we identify functions that agree almost everywhere.

Let $\mathcal{C}_{n}$ be either the collection $\mathcal{Q}_{n}$ or $\mathcal{R}_{n}$ of all cubes or rectangles, respectively, in $\IR^{n}$ of positive and finite measure with sides parallel to the coordinate axex.
Let $M_{\mathcal{C}_{n}}$ denote the associated Hardy--Littlewood maximal function operator. For a weight $\omega$ on $\IR^{n}$ and a Borel set $A \subset \IR^{n}$, we write
\[
\omega(A) = \int_{A}\omega(x)\,dx \in [0,\infty].
\]
The $p$-dual weight of $\omega$ is the weight $\omega'_{p} := \omega^{-\frac{1}{p-1}}$, where $p \in (1,\infty)$.
We define the $\mathcal{A}_{p}$-characteristics
\[
[\omega,\sigma]_{\mathcal{A}_{p}(\IR^{n};\mathcal{C}_{n})} := \sup_{A \in \mathcal{C}_{n}}\frac{\omega(A)}{|A|}\left(\frac{\sigma'_{p}(A)}{|A|}\right)^{p-1}, \qquad p \in (1,\infty),
\]
\[
[\omega]_{\mathcal{A}_{p}(\IR^{n};\mathcal{C}_{n})} := [\omega,\omega]_{\mathcal{A}_{p}(\IR^{n};\mathcal{C}_{n})} = \sup_{A \in \mathcal{C}_{n}}\frac{\omega(A)}{|A|}\left(\frac{\sigma'_{p}(A)}{|A|}\right)^{p-1}, \qquad p \in (1,\infty),
\]
and
\[
[\omega]_{\mathcal{A}_{\infty}(\IR^{n};\mathcal{C}_{n})} := \sup_{A \in \mathcal{C}_{n}} \frac{1}{\omega(A)}\int_{A}M_{\mathcal{C}_{n}}(\omega \mathds{1}_{A})\,dx.
\]
For $p \in (1,\infty]$ the Muckenhoupt class $\mathcal{A}_{p}(\IR^{n};\mathcal{C}_{n})$ is defined as the set of all weights $\omega$ on $\IR^{n}$ with $[\omega]_{\mathcal{A}_{p}(\IR^{n};\mathcal{C}_{n})} < \infty$.
For $p \in (1,\infty)$ it holds that $\omega \in \mathcal{A}_{p}(\IR^{n};\mathcal{C}_{n})$ if and only if $\omega'_{p} \in \mathcal{A}_{p'}(\IR^{n};\mathcal{C}_{n})$, in which case $[\omega]_{\mathcal{A}_{p}(\IR^{n};\mathcal{C}_{n})} = [\omega'_{p}]_{\mathcal{A}_{p'}(\IR^{n};\mathcal{C}_{n})}^{p-1}$.
For $1 < p_{0} \leq p_{1} \leq \infty$ it holds that $\mathcal{A}_{p_{1}}(\IR^{n};\mathcal{C}_{n}) \subset \mathcal{A}_{p_{0}}(\IR^{n};\mathcal{C}_{n})$ with $1 \leq [\omega]_{\mathcal{A}_{p_{1}}(\IR^{n};\mathcal{C}_{n})} \leq [\omega]_{\mathcal{A}_{p_{0}}(\IR^{n};\mathcal{C}_{n})}$.

If $\omega$ is an $\mathcal{A}_p$-weight,
\[
\mathcal{S}(\IR^n;X) \stackrel{d}{\hookrightarrow} L^p_{\omega}(\IR^n;X) \hookrightarrow \mathcal{S}'(\IR^{n};X).
\]
One therefore may ask under which conditions on a function $m \in L^{\infty}(\IR^{n};\mathcal{B}(X,Y))$ the operator
		\begin{equation}\label{eq:t_m_basic_setting}
			T_m\colon \mathcal{S}(\IR^n;X) \ni f \mapsto \mathcal{F}^{-1}(\xi \mapsto m(\xi)(\mathcal{F}f)(\xi)) \in \mathcal{S}'(\IR^n;Y)
		\end{equation}
	induces a bounded operator $L^p_{\sigma}(\IR^n;X) \to L^p_{\omega}(\IR^n;Y)$. In this case we say that $m$ is a \emph{bounded multiplier}. We denote by $\mathcal{M}_p^n((X,\sigma) \to (Y,\omega))$ the space of all such bounded multipliers and write $\mathcal{M}_p^n(X,\omega)$ if both $X$ and $Y$ and $\sigma$ and $\omega$ agree. Its norm is given by the operator norm of the Fourier multiplier operator.

For a Borel measurable set $A \subset \IR^{n}$ we use the following special notation for the Fourier multiplier with as symbol the associated indicator function $\mathds{1}_{A}$: $\Delta(A) := T_{\mathbf{1}_{A}}$.

The pairing
\[
L^{p}_{\omega}(\IR^{n};X) \times L^{p'}_{\omega'_{p} }(\IR^{n};X^{*}) \longrightarrow \IC,\: \int_{\IR^{n}}(f,g) \mapsto \langle f , g \rangle \,d\lambda,
\]
is norming.
Under this pairing one has $[L^{p}_{\omega}(\IR^{n};X)]^{*} = L^{p'}_{\omega'_{p}}(\IR^{n};X^{*})$ when $X$ is e.g.\ reflexive.

If $\omega,\sigma \in \mathcal{A}_{p}$, then $\omega'_{p},\sigma'_{p} \in \mathcal{A}_{p'}$ and
\[
\mathcal{M}_p^n\left((X,\sigma) \to (Y,\omega)\right) \longrightarrow \mathcal{M}_{p'}^n\left((Y^{*},\sigma'_{p}) \to (X^{*},\omega'_{p} )\right),\: m \mapsto \widetilde{m}^{*},
\]
defines an isometric isomorphism, where $\widetilde{m}^{*}(\xi) = [m(-\xi)]^{*}$ and $T_{\widetilde{m}^{*}}$ is obtained from $T_{m^{*}}$ by restriction.

Reasonable multiplier theorems cannot be obtained on arbitrary Banach spaces as even the most basic multiplier, namely the vector-valued Hilbert transform
		\begin{equation}\label{eq:hilbert_transform}
			(Hf)(x) \coloneqq \lim_{\epsilon \downarrow 0} \int_{\abs{x-t} \ge \epsilon} \frac{f(t)}{x - t} \d t ,
		\end{equation}
	does not give rise to a bounded operator $L^2(\IR;X) \to L^2(\IR;X)$ for arbitrary Banach spaces $X$. This leads to the following definition.
	
	\begin{definition}\label{def:hilbert_transform}
		A Banach space $X$ is said to be of class \emph{HT} if the vector-valued Hilbert transform~\eqref{eq:hilbert_transform} initially defined on $\mathcal{S}(\IR;X)$ induces a bounded operator $L^p(\IR;X) \to L^p(\IR;X)$ for one or equivalently (by Hörmander's condition) all $p \in (1, \infty)$.
	\end{definition}

Recall that the Hilbert transform can be realized as the Fourier multiplier operator with symbol $\imath \sgn$.
A a consequence, $X$ is of class \emph{HT} if and only if $\Delta(\IR) \in \mathcal{B}(L^{p}(\IR;X))$ (i.e.\ $\mathds{1}_{\IR_{+}} \in \mathcal{M}^{1}_{p}(X)$) for some/all $p \in (1,\infty)$. 	

A deep result due to Burkholder and Bourgain (\cite[Theorem~5.1.1]{HNVW16}) says that a Banach space $X$ is of class HT if and only if $X$ is a \emph{UMD space}. UMD is a primarily probabilistic notion and stands for unconditionality of martingale differences (\cite[Definition~4.2.1]{HNVW16}).

	For example, all reflexive $L^p$-spaces are UMD spaces (\cite[Proposition~4.2.15]{HNVW16}). One can show that on UMD spaces the Mikhlin multiplier theorem holds for scalar-valued multipliers, see for example~\cite[Theorem~5.5.10]{HNVW16}. For operator-valued multipliers norm boundedness must be replaced by $\mathcal{R}$-boundedness.

A Rademacher sequence is a sequence of independent random variables $(\epsilon)_{k \in \IN}$ on some probability space $(\Omega,\mathcal{F},\IP)$ with $\IP(\epsilon_k = \pm 1) = \frac{1}{2}$ for all $k \in \IN$. In the following we fix a Rademacher sequence $(\epsilon)_{k \in \IN}$.
	
	\begin{definition}
		A subset $\mathcal{T} \subset \mathcal{B}(X,Y)$ is called \emph{$\mathcal{R}$-bounded} if there exists a constant $C \ge 0$ such that for all $n \in \IN$, $T_1, \ldots, T_n \in \mathcal{T}$, $x_1, \ldots, x_n \in X$ one has
			\[
				\biggnorm{\sum_{k=1}^n \epsilon_k T_k x_k}_{L^2(\Omega;Y)} \le C \biggnorm{\sum_{k=1}^n \epsilon_k x_k}_{L^2(\Omega;X)}.
			\]
		The smallest constant for which the above inequality holds is denoted by $\mathcal{R}(\mathcal{T})$.
	\end{definition}
	
	For the basic permanence properties of $\mathcal{R}$-boundedness under sums, compositions and unions, which will be used in the following, we refer to~\cite[Section~8.1]{HNVW17}. We are now ready to formulate the Mikhlin theorem for operator-valued Fourier multipliers in the unweighted case (\cite[Theorem~5.5.10]{HNVW16}).
	
	\begin{theorem}\label{thm:mikhlin}
		Let $X$ and $Y$ be UMD spaces and $m \in C^n(\IR^n \setminus \{0\}; \mathcal{B}(X,Y))$. Suppose that
			\[
				\sup_{\abs{\alpha}_{\infty} \le 1} \mathcal{R} \{\abs{\xi}^{\abs{\alpha}} \partial^{\alpha} m(\xi): \xi \in \IR^n \setminus \{ 0 \} \} < \infty.
			\]
		Then $m$ is a bounded Fourier multiplier, i.e.\ $m \in \mathcal{M}^n_p((X,\mathds{1}) \to (Y,\mathds{1}))$, for all $p \in (1, \infty)$. More precisely, there exists a constant $C > 0$ only depending on $n$, $p$, $X$ and $Y$ such that
			\[
				\norm{T_m} \le C \sup_{\abs{\alpha}_{\infty} \le 1} \mathcal{R} \{\abs{\xi}^{\abs{\alpha}} \partial^{\alpha} m(\xi): \xi \in \IR^n \setminus \{ 0 \} \}.
			\]
	\end{theorem}
				
	The case of scalar-valued multipliers is contained in the above result. Indeed, by Kahane's contraction principle (\cite[Proposition~3.2.10]{HNVW17}) the set $\{ c \Id: \abs{c} \le 1 \}$ is $\mathcal{R}$-bounded in every Banach space.

\subsection{Extrapolation of Calderón--Zygmund operators}\label{sec:extrapolation_CZ}

	We will obtain a smooth variant of the Littlewood--Paley estimate as a consequence of extrapolation results for Calderón--Zygmund operators. In this section we present the necessary background in a smooth setting sufficient for our needs. For Banach spaces $X$ and $Y$, a Bochner measurable function $K : \IR^{n} \setminus \{0\} \to \mathcal{B}(X,Y)$ is called a \emph{Calderón--Zygmund kernel (of convolution type)} if, for some constant $C>0$,
\begin{enumerate}
\item it obeys the decay estimate
\begin{align*}
\norm{K(x)}_{\mathcal{B}(X,Y)} \le C \abs{x}^{-n}, \qquad x \neq 0,
\end{align*}
\item and it obeys the H\"older type estimate
\[
\norm{K(x-y)-K(x)}_{\mathcal{B}(X,Y)} \le C \abs{y}^{\alpha}\abs{x}^{-n-\alpha}, \qquad 0 < \abs{y} < \frac{1}{2}\abs{x-y}
\]
for some H\"older exponent $\alpha \in (0,1]$.
\end{enumerate}

		A bounded operator $T\colon L^p(\IR^n;X) \to L^p(\IR^n;Y)$ is called a \emph{Calderón--Zygmund operator} if there exists a Calderón--Zygmund kernel $K$ such that for all $f \in C^{\infty}_c(\IR^n;X)$ and almost all $x \not\in \supp f$ one has the representation
	\[
		(Tf)(x) = \int_{\IR^n} K(x - y) f(y) \d y.
	\]	
	We use the following extrapolation result for Calderón--Zygmund operators, which is a reformulation of \cite[Corollary~3.3]{HanHyt14b} (with $\sigma^{-\frac{1}{p-1}}$ playing the role of $\sigma$ there).

	\begin{theorem}\label{thm:extrapolation}
		Let $X$ be a Banach space, $T$ a Calderón--Zygmund operator on $L^p(\IR^n; X)$ and $\sigma$, $\omega$ such that $\omega, \sigma^{-\frac{1}{p-1}} \in \mathcal{A}_\infty(\IR^n;\mathcal{Q}_n)$ and $[\omega,\sigma]_{\mathcal{A}_p(\IR^n;\mathcal{Q}_n)} < \infty$. Then $T$ induces a bounded operator $L^p_{\sigma}(\IR^n;X) \to L^p_{\omega}(\IR^n;X)$ with
			\begin{equation*}
				\norm{T}_{L^p_{\sigma} \to L^p_{\omega}} \lesssim [\omega, \sigma]_{\mathcal{A}_p(\IR^n; \mathcal{Q}_n)}^{1/p} \left([\omega]_{\mathcal{A}_\infty(\IR^n;\mathcal{Q}_n)}^{1-\frac{1}{p}} + [\sigma^{-\frac{1}{p-1}}]_{\mathcal{A}_\infty(\IR^n;\mathcal{Q}_n)}^{\frac{1}{p}}\right).
			\end{equation*}
		The implicit constant only depends on $\norm{T}_{L^{p} \to L^{p}}$, $p$, the dimension $n$ and the constant $C$ in the definition of a Calderón--Zygmund kernel. In particular, if $\omega = \sigma$, then
			\begin{equation*}
				\norm{T}_{L^p_{\omega} \to L^p_{\omega}} \lesssim [\omega]_{\mathcal{A}_p(\IR^n;\mathcal{Q}_n)}^{\max \{ 1, \frac{1}{p-1} \}}.
			\end{equation*}
	\end{theorem}

\subsection{Unconditional Decompositions}

In this subsection we recall some facts from the theory of unconditional (Schauder) decompositions, with references \cite{HNVW16}, \cite{HNVW17}, \cite{CPSW00} and \cite{Wit00}. We take the setting from \cite[Section~4.1.b]{HNVW16} on unconditional decompositions, which in the context of Littlewood--Paley decompositions provides a more natural framework than that of Schauder decompositions.

Given an index set $I$, we denote by $\{\epsilon_i\}_{i \in I}$ a family of independent identically distributed random variables on some probability space $(\Omega, \mathcal{F}, \IP)$ with $\IP(\epsilon_i = \pm 1) = \frac{1}{2}$. In case $I=\IN$ we get a Rademacher sequence.

A \emph{pre-decomposition} of a Banach space $X$ is a family of bounded linear projections $\Delta =(\Delta_{i})_{i \in I}$ in $X$ with the property that $\Delta_{i}\Delta_{j} = 0$ whenever $i \neq j$.
An \emph{unconditional Schauder decompostion} of $X$ is a pre-decomposition $\Delta =(\Delta_{i})_{i \in I}$ of $X$ with the property that $x=\sum_{i \in I}\Delta_{i}x$ in $X$ for all $x \in X$. A \emph{Schauder decomposition} of $X$ is a pre-decomposition $\Delta =(\Delta_{i})_{i \in \IN}$ of $X$ with the property that $x=\sum_{i=0}^{\infty}\Delta_{i}x$ in $X$ for all $x \in X$. Note that every unconditional decomposition $\Delta =(\Delta_{i})_{i \in I}$ of $X$ with $I$ countably infinite can be realized as Schauder decomposition by any enumeration of $I$.
A family $D=(D_{i})_{i \in I} \subset \mathcal{B}(X)$ is called \emph{U$^{+}$} if there exists a finite constant $C^{+}>0$ such that
\[
\biggnorm{\sum_{i \in F}\epsilon_{i}D_{i}x}_{L^{2}(\Omega;X)} \leq C^{+}\biggnorm{\sum_{i \in F}D_{i}x}_{X}
\]
for all finite subsets $F \subset I$ and $x \in X$, and \emph{U$^{-}$} if there exists a finite constant $C^{-}>0$ with
\[
\biggnorm{\sum_{i \in F}D_{i}x}_{X} \leq C^{-} \biggnorm{\sum_{i \in F}\epsilon_{i}D_{i}x}_{L^{2}(\Omega;X)}
\]
for all finite subsets $F \subset I$ and $x \in X$. We denote the smallest such constants $C^{+}>0$ and $C^{-}>0$ by $C^{+}_{D}>0$ and $C^{-}_{D}>0$, respectively.
Let $\Delta = (\Delta_{i})_{i \in I} \subset \mathcal{B}(X)$. For each finite subset $F \subset I$ we define $\Delta_{F}:= \sum_{i \in F}\Delta_{i}$. We futhermore define
\[
\mathrm{ran}(\Delta) := \cup\left\{ \Delta_{F}(X) : F \subset I \:\text{finite}\right\}.
\]

\begin{lemma}\label{lem:basic_char_uncond_Schaud_decomp}
For a pre-decomposition $\Delta =(\Delta_{i})_{i \in I}$ of a Banach space $X$ with $\mathrm{ran}(\Delta)$ dense in $X$ the following are equivalent:
\begin{equiv_enum}
	\item $\Delta$ is an unconditional decomposition.
	\item\label{item:lemma:undcond_char} There exists a finite constant $C>0$ such that for all $(\varepsilon_{i})_{i \in I} \in \{-1,1\}^{I}$, finite subsets $F \subset I$ and $x \in X$
	\[
	\biggnorm{\sum_{i \in F}\varepsilon_{i}\Delta_{i}x}_{X} \leq C \biggnorm{\sum_{i \in F} \Delta_{i}x}_{X}.
	\]

	\item $\Delta$ is \emph{U$^{+}$} and \emph{U$^{-}$}.
\end{equiv_enum}
The smallest admissible constant $C$ in~\ref{item:lemma:undcond_char} is called the \emph{unconditional constant} of $\Delta$ and is denoted by $C_{\Delta}$. Moreover, it holds that $C^{-}_{\Delta},C^{+}_{\Delta} \leq C_{\Delta} \leq C^{-}_{\Delta}C^{+}_{\Delta}$.
\end{lemma}

Using the characterization of unconditional decompositions in terms of \emph{U$^{+}$} and \emph{U$^{-}$} one can establish the following abstract multiplier theorem~\cite[Theorem~3.4]{CPSW00}.

\begin{theorem}\label{thm:abstract_multiplier}
		Let $X$ and $Y$ be Banach spaces and $\Delta^{X} = (\Delta_i^X)_{i \in I}$, $\Delta^{Y} = (\Delta_i^Y)_{i \in I}$ unconditional decompositions of $X$ and $Y$, respectively. Further suppose that $(M_i)_{i \in I} \subset \mathcal{B}(X,Y)$ is $\mathcal{R}$-bounded with $\Delta_i^Y M_i = \Delta_i^Y M_i \Delta_i^X$ for all $i \in I$. Then
			\[
				Mx \coloneqq \sum_{i \in I} M_i \Delta_i x
			\]
		is summable for all $x \in X$ and defines a bounded linear operator $M\colon X \to Y$ with
			\begin{equation*}
				\norm{M} \leq C_X^{+} C_Y^{-} \mathcal{R}\{M_i: i \in I\}.
			\end{equation*}
	\end{theorem}

For Banach spaces $X$ and $Y$ that have \emph{Pisier's property $(\alpha)$} there is a useful $\mathcal{R}$-boundedness version of the above theorem. Before we state it, let us first recall \emph{Pisier's property $(\alpha)$}.
Let $(\epsilon_{i}')_{i \geq 1}$ and $(\epsilon_{j}'')_{j \geq 1}$ be independent Rademacher sequences on probability
spaces $(\Omega',\mathcal{F}',\IP')$ and $(\Omega'',\mathcal{F}'',\IP'')$, respectively.
A Banach space $X$ is said to have \emph{Pisier's property $(\alpha)$} (or \emph{Pisier's contraction property}) if there exists a finite constant $C > 0$ such that
		\begin{equation*}
		 \biggnorm{\sum_{i=1}^{M}\sum_{j=1}^{N} a_{i,j}\epsilon_i' \epsilon''_j x_{i,j}}_{L^{2}(\Omega'\times\Omega'';X)} \le C|a|_{\infty} \biggnorm{\sum_{i=1}^{M}\sum_{j=1}^{N} \epsilon'_{i}\epsilon''_{j} x_{i,j}}_{L^{2}(\Omega'\times\Omega'';X)}
		\end{equation*}
for all $M,N \in \IN$, $a=(a_{i,j})_{1 \le i \le M,1 \le j \le N} \subset \IC$ and $(x_{i,j})_{1 \le i \le M,1 \le j \le N} \subset X$.
The smallest such constant is denoted by $\alpha_{X}$.

\begin{theorem}\label{thm:abstract_multiplier;prop_alpha}\cite[Theorem~3.14]{CPSW00}
Let $X$ and $Y$ be Banach spaces with Pisier's property~$(\alpha)$ and $\Delta^{X} = (\Delta_i^X)_{i \in I}$, $\Delta^{Y} = (\Delta_i^Y)_{i \in I}$ unconditional decompositions of $X$ and $Y$, respectively.
Let $\mathcal{M} \subset \mathcal{B}(X,Y)$ be an $\mathcal{R}$-bounded collection of operators and
\[
\mathcal{T} \coloneqq \left\{ \sum_{i \in I}M_{i}\Delta_{i} : M_{i} \in \mathcal{M} \:\mbox{such that $\Delta_i^Y M_i = \Delta_i^Y M_i \Delta_i^X$ for all $i \in I$} \right\} \subset \mathcal{B}(X,Y).
\]
Then $\mathcal{T}$ is $\mathcal{R}$-bounded with
\begin{equation*}
\mathcal{R}(\mathcal{T}) \leq \alpha_{X}C_{X}^{+}\alpha_{Y}C_{Y}^{-}\mathcal{R}(\mathcal{M}).
\end{equation*}
	\end{theorem}
Note that $\mathcal{T}$ is well-defined by Theorem~\ref{thm:abstract_multiplier}. In the setting of Littlewood--Paley decompositions it is convenient to use duality in order to verify that a family of spectral projections forms an unconditional decomposition, the adjoint family being of the same form.
The following proposition provides the abstract basis for such a duality argument.

\begin{proposition}\label{prop:uncond_Schaud_bbd_into_Rad}
Let $\Delta =(\Delta_{i})_{i \in I}$ be a pre-decomposition of $X$ with adjoint family $\Delta^{*} =(\Delta_{i}^{*})_{i \in I}$. Then $\Delta$ is an unconditional decomposition if $\mathrm{ran}(\Delta)$ is dense in $X$ and both $\Delta$ and $\Delta^{*}$ are \emph{U$^{+}$}. Moreover, in this situation we have (in addition to Lemma~\ref{lem:basic_char_uncond_Schaud_decomp}) $C^{-}_{\Delta} \leq C^{+}_{\Delta^{*}}$.
\end{proposition}

\subsection{A Generic Fourier Multiplier Theorem}

In this subsection we follow the approach presented in the survey article~\cite{Hyt07c}, which was concerned with the unweighted setting, to obtain Fourier multiplier theorems out of Littlewood--Paley decompositions. This is basically the usual approach but put in a nice abstract framework that cleans up the arguments.
As no proofs are given in~\cite{Hyt07c}, we have decided to include those here in quite some detail in order to make the paper more accessible.

For the rest of this section, let
$X$, $Y$, $E$ and $F$ be Banach spaces with
\[
\mathcal{S}(\IR^{n};X) \stackrel{d}{\hookrightarrow} E \stackrel{d}{\hookrightarrow} \mathcal{S}'(\IR^{n};X),\qquad F \subset \mathcal{S}'(\IR^{n};X), \qquad \mathcal{S}(\IR^{n};Y) \stackrel{d}{\hookrightarrow}  G \stackrel{d}{\hookrightarrow} \mathcal{S}'(\IR^{n};Y),
\]
$\mathcal{S}(\IR^{n};Y^{*}) \stackrel{d}{\subset} G^{*}$, $\mathcal{B}(X,Y) \hookrightarrow \mathcal{B}(F,G)$ contractively by pointwise multiplication and
\[
M:= \mathcal{R}\{ \Delta([\eta,\infty)) : \eta \in \IR^{n} \} < \infty \quad \text{in} \quad \mathcal{B}(E,F).
\]
Here
\[
\mathcal{S}(\IR^{n};Y^{*}) = [\mathcal{S}'(\IR^{n};Y)]' \hookrightarrow  G^{*} \hookrightarrow [\mathcal{S}(\IR^{n};Y)]' = \mathcal{S}'(\IR^{n};Y^{*})
\]
under the natural identifications; $\mathcal{S}(\IR^{n};Y^{*}) \stackrel{d}{\subset} G^{*}$ holds for instance when $G$ is reflexive.

We denote by $\mathcal{M}(E \to G)$ the space of all Fourier multiplier symbols with $T_{m} \in \mathcal{B}(E,G)$ equipped with the natural norm.

\begin{definition}\label{def:uniform_R-bdd_var}
We say that a set of functions $\mathscr{M} \subset L^{\infty}(\IR^{n};\mathcal{B}(X,Y))$ is of \emph{uniformly $\mathcal{R}$-bounded variation} if there exist a constant $C>0$, an $\mathcal{R}$-bounded set $\mathscr{T} \subset \mathcal{B}(X,Y)$, and for each $m \in \mathscr{M}$ a Borel measure $\mu_{m}$ on $\IR^{n}$ and a bounded WOT-measurable function $\tau_{m}:\IR^{n} \to \mathcal{B}(X,Y)$, with
$\norm{\mu_{m}} \leq C$ and $\tau_{m}(\IR^{n}) \subset \mathscr{T}$, such that
\begin{equation}\label{eq:def:uniform_R-bdd_var}
\ip{m(\xi)x}{y^{*}} = \int_{(-\infty,\xi]}\ip{\tau_m(\eta)x}{y^{*}} \ud \mu_m(\eta), \qquad \xi \in \IR^{n}, x \in X, y^{*} \in Y^{*}.
\end{equation}
We define
\[
\mathrm{var}_{\mathcal{R}}(\mathscr{M}) := \inf\{ \,C\mathcal{R}(\mathscr{T}) : C>0, \mathscr{T} \subset \mathcal{B}(X,Y)\:\:\text{as above}  \}.
\]
\end{definition}

\begin{lemma}\label{lemma:multiplier_rep_via_Borel_measure_rep}
Suppose that $m \in L^{\infty}(\IR^{n};\mathcal{B}(X,Y))$ can be represented as in \eqref{eq:def:uniform_R-bdd_var} for some complex Borel measure $\mu$ on $\IR^{n}$ and a bounded WOT-measurable function $\tau\colon \IR^{n} \to \mathcal{B}(X,Y)$.
Then we have $m \in \mathcal{M}(E \to G)$ with
\[
\norm{m}_{\mathcal{M}(E \to G)} \leq \sup\left\{ \norm{\Delta([\eta,\infty))}_{\mathcal{B}(E,F)} : \eta \in \IR^{n} \right\}\,\norm{\tau}_{\infty}\norm{\mu}.
\]
Moreover, for every $f \in E$ and $g \in G^{*}$,
$\IR^{n} \owns \eta \mapsto \ip{g}{\tau(\eta)\Delta([\eta,\infty))f}_{\ip{G}{G^{*}}} \in \IC$ is a bounded Borel measurable function from which the Fourier multiplier operator $T_{m} \in \mathcal{B}(E,G)$ can be obtained by
\begin{equation}\label{eq:lemma:multiplier_rep_via_Borel_measure_rep}
\ip{T_{m}f}{g}_{\ip{G}{G^{*}}} = \int_{\IR^{n}}\ip{\tau(\eta)\Delta([\eta,\infty))f}{g}_{\ip{G}{G^{*}}} \d\mu(\eta).
\end{equation}
\end{lemma}
\begin{proof}
We put $C:=\sup\left\{ \norm{\Delta([\eta,\infty))}_{\mathcal{B}(E,F)} : \eta \in \IR^{n} \right\} \leq M < \infty$. Note that, as a consequence of the assumptions, $G$ can be described in terms of $G^{*}$ as follows:
\begin{equation}\label{eq:proof:lemma:multiplier_rep_via_Borel_measure_rep;G}
G = \left\{ u \in \mathcal{S}'(\IR^{n};Y) : [g \mapsto \ip{u}{g}_{\ip{\mathcal{S'}}{\mathcal{S}}}] \in (\mathcal{S}(\IR^{n};Y^{*}),\norm{\,\cdot\,}_{G^{*}})^{*} \right\} \quad \text{isometrically}.
\end{equation}

Let us first prove the measurabilty of $\IR^{n} \owns \eta \mapsto \ip{g}{\tau(\eta)\Delta([\eta, \infty))f}_{\ip{G}{G^{*}}} \in \IC$ for every $f \in E$ and $g \in G^{*}$.
For each fixed $\eta \in \IR^{n}$ it holds that
\[
E \times G^{*} \to \IC,\, (f,g) \mapsto \ip{\tau(\eta)\Delta([\eta,\bar{\infty}))f}{g}_{\ip{G}{G^{*}}}
\]
is a continuous bilinear map, satisfying the bound
\begin{equation}\label{eq:proof:lemma:multiplier_rep_via_Borel_measure_rep}
|\ip{\tau(\eta)\Delta([\eta,\infty))f}{g}_{\ip{G}{G^{*}}}| \leq C \,\norm{\tau}_{\infty}\norm{f}_{E}\norm{g}_{G^{*}}.
\end{equation}
Since
\[
\mathcal{S}(\IR^{n}) \otimes X \stackrel{d}{\subset} \mathcal{S}(\IR^{n};X) \stackrel{d}{\hookrightarrow} E,\qquad \mathcal{S}(\IR^{n}) \otimes Y^{*} \stackrel{d}{\subset} \mathcal{S}(\IR^{n};Y^{*}) \stackrel{d}{\hookrightarrow} G^{*},
\]
it thus is enough to consider $f = \phi \otimes x$ and
$g=\psi \otimes y^{*}$ with $\phi \in \mathcal{S}(\IR^{n})$, $\psi \in \mathcal{S}(\IR^{n})$, $x \in X$,
and $y^{*} \in Y^{*}$. Then $\eta \mapsto \ip{\tau(\eta)\Delta([\eta,\infty))f}{g}_{\ip{G}{G^{*}}} = \ip{\Delta([\eta,\infty))\phi}{\psi}_{\ip{\mathcal{S}'}{\mathcal{S}}}\ip{\tau(\eta)x}{y^{*}}_{\ip{Y}{Y^{*}}}$ is measurable, being the product of two measurable functions.

Since the measurable function $y \mapsto \ip{\tau(\eta)\Delta([\eta,\infty))f}{g}_{\ip{G}{G^{*}}}$ satisfies the bound \eqref{eq:proof:lemma:multiplier_rep_via_Borel_measure_rep},
it follows that the expression on the right hand-side of \eqref{eq:lemma:multiplier_rep_via_Borel_measure_rep} is well defined and, in fact, gives rise to a bounded bilinear form
\begin{equation*}\label{eq:lemma;multiplier_rep_via_Borel_measure_rep;bdd_bil_form}
B_{\tau,\mu}\colon E \times G^{*} \to \IC,\,
(f,g) \mapsto \int_{\IR^{n}}\ip{\tau(\eta)\Delta([\eta,\infty))f}{g}_{\ip{G}{G^{*}}} \d\mu(\eta)
\end{equation*}
of norm $\leq C \,\norm{\tau}_{\infty}\norm{\mu}$.
Since $\mathcal{S}(\IR^{n};X)$ and $\mathcal{S}(\IR^{n};Y^{*})$ are dense in
$E$ and $G^{*}$, respectively, in view of this bound for $B_{\tau,\mu}$ and \eqref{eq:proof:lemma:multiplier_rep_via_Borel_measure_rep;G},
it is thus enough to show $\ip{T_{m}f}{g}_{\ip{\mathcal{S}'}{\mathcal{S}}} = B_{\tau,\mu}(f,g)$
holds for all $f \in \mathcal{S}(\IR^{n};X)$ and $g \in \mathcal{S}(\IR^{n};Y^{*})$; here $T_{m}$ is at this moment of course still the operator~\eqref{eq:t_m_basic_setting}.
So let $f \in \mathcal{S}(\IR^{n};X)$ and $g \in \mathcal{S}(\IR^{n};Y^{*})$.
Then
\begin{align*}
\ip{T_{m}f}{g}_{\ip{\mathcal{S}'}{\mathcal{S}}}
&= \ip{m\hat{f}}{\check{g}}_{\ip{\mathcal{S}'}{\mathcal{S}}} = \ip{m\hat{f}}{\check{g}}_{\ip{L^{\infty}}{L^{1}}}
= \int_{\IR^{n}}\int_{(-\infty,\xi]}\ip{\tau(\eta)\hat{f}(\xi)}{\check{g}(\xi)} \d\mu(\eta) \d\xi \\
&= \int_{\IR^{n}}\int_{[\eta,\infty)}\ip{\tau(\eta)\hat{f}(\xi)}{\check{g}(\xi)}\ d\xi \d\mu(\eta) \\
&= \int_{\IR^{n}}\int_{\IR^{n}}\ip{\tau(\eta)\mathds{1}_{[\eta,\infty)}(\xi)\hat{f}(\xi)}{\check{g}(\xi)} \d\xi \d\mu(\eta) \\
&= \int_{\IR^{n}}\ip{\tau(\eta)\mathds{1}_{[\eta,\infty)}\hat{f}}{\check{g}}_{\ip{L^{\infty}}{L^{1}}} \d\mu(\eta) = \int_{\IR^{n}}\ip{\tau(\eta)\mathds{1}_{[\eta,\infty)}\hat{f}}{\check{g}}_{\ip{\mathcal{S}'}{\mathcal{S}}} \d\mu(\eta) \\
&= \int_{\IR^{n}}\ip{\tau(\eta)\Delta([\eta,\infty))f}{g}_{\ip{\mathcal{S}'}{\mathcal{S}}} \d\mu(\eta)
= \int_{\IR^{n}}\ip{\tau(\eta)\Delta([\eta,\infty))f}{g}_{\ip{G}{G^{*}}} \d\mu(\eta) \\
&= B_{\tau,\mu}(f,g),
\end{align*}
where we used Fubini's theorem in the fourth equality.
\end{proof}

\begin{proposition}\label{prop:R-bbd_fm_from_unif_R-bdd_var}
If $\mathscr{M} \subset L^{\infty}(\IR^{n};\mathcal{B}(X,Y))$ is of uniformly $\mathcal{R}$-bounded variation, then
\[
\mathcal{R}\{ T_{m} : m \in \mathscr{M} \} \leq M\,\mathrm{var}_{\mathcal{R}}(\mathscr{M}) \qquad \text{in}\:\: \mathcal{B}(E,G).
\]
\end{proposition}

\begin{proof}
Let $C>0$ and $\mathscr{T}$ be as in the definition of uniformly $\mathcal{R}$-bounded variation for $\mathscr{M}$, and define
$\mathscr{S} := \{ \Delta([\eta, \infty)) : \eta \in \IR^{n} \} \subset \mathcal{B}(E,F)$.
Each $m \in \mathscr{M}$ in particular satisfies the hypotheses of Lemma~\ref{lemma:multiplier_rep_via_Borel_measure_rep}, whence the associated Fourier multiplier operator $T_{m} \in \mathcal{B}(E,G)$ has the representation \eqref{eq:lemma:multiplier_rep_via_Borel_measure_rep}. Modifying the argument in \cite[Theorem~8.5.2]{HNVW17} from the strong operator topology to the weak operator topology, this representation yields
\[
T_{m} \in  C \,\overline{\mathrm{abs}\,\mathrm{conv}}(\mathscr{T}\mathscr{S}),\qquad m \in \mathscr{M},
\]
where the closure is taken in the weak operator topology on $\mathcal{B}(E,G)$.
Then, by the basic stability properties of $\mathcal{R}$-bounds (see \cite[Section~8.1.e]{HNVW17}),
\[
\mathcal{R}\{ T_{m} : m \in \mathscr{M} \} \leq C\mathcal{R}(\mathscr{T})\mathcal{R}(\mathscr{S}).
\]
Using the $\mathcal{R}$-boundedness assumption $M=\mathcal{R}(\mathscr{S})<\infty$ and taking the infimum over all admissible $C>0$ and $\mathscr{T}$ gives the desired result.
\end{proof}

Combining Proposition~\ref{prop:R-bbd_fm_from_unif_R-bdd_var} with Theorem~\ref{thm:abstract_multiplier}/\ref{thm:abstract_multiplier;prop_alpha} we arrive at the following generic Fourier multiplier theorem:

\begin{theorem}\label{thm:generic_fm_thm}
Let $\mathscr{J} \subset \mathcal{R}_{n}$ be a countable collection of rectangles for which $\Delta_{E} = \{ \Delta[R] : R \in \mathscr{J} \} \subset \mathcal{B}(E)$ and $\Delta_{F} = \{ \Delta[R] : R \in \mathscr{J} \} \subset \mathcal{B}(G)$ form unconditional decompositions of $E$ and $G$, respectively.
		\begin{thm_enum}%
			\item\label{item:thm:generic_fm_thm;general} If $m \in L^{\infty}(\IR^{d};\mathcal{B}(X,Y))$ is a symbol with the property that $\{ m\mathds{1}_{J}: J \in \mathscr{J} \}$ is of uniformly $\mathcal{R}$-bounded variation, then we have $m \in \mathcal{M}(E \to G)$ with
\begin{align*}
\norm{m}_{\mathcal{M}(E \to G)} \leq M\,C^{+}_{\Delta_{E}} C^{-}_{\Delta_{F}}\,\mathrm{var}_{\mathcal{R}}(\{ m1_{J}: J \in \mathscr{J} \}).
\end{align*}
			\item\label{item:thm:generic_fm_thm;alpha} Suppose additionally that $E$ and $G$ have Pisier's property $(\alpha)$. Suppose that $\mathscr{M} \subset L^{\infty}(\IR^{d};\mathcal{B}(X,Y))$ is a set of symbols such that $\{ m\mathds{1}_{J}: m \in \mathscr{M}, J \in \mathscr{J} \}$ is of uniformly $\mathcal{R}$-bounded variation. Then one has, in $\mathcal{B}(E,G)$,
\begin{align*}
\mathcal{R}\{T_{m} : m \in \mathscr{M} \} \leq M\,\alpha_{E}\alpha_{F}\,C^{+}_{\Delta_{E}} C^{-}_{\Delta_{F}}\,\mathrm{var}_{\mathcal{R}}(\{ m1_{J}: m \in \mathscr{M}, J \in \mathscr{J} \}).
\end{align*}
		\end{thm_enum}
\end{theorem}
\begin{proof}
We only need to show the compatibility between the Fourier multiplier operator $T_{m}$ and the operator, say $D_{m}$, obtained from the abstract multiplier result Theorem~\ref{thm:abstract_multiplier}/\ref{thm:abstract_multiplier;prop_alpha}.
So let $f \in \mathcal{S}(\IR^{n};X)$. It is enough to show that $T_{m}f=D_{m}f$.
Writing $m_{J}:=m\mathds{1}_{J}$, we have $D_{m} = \sum_{J \in \mathscr{J}}T_{m_{J}}\Delta_{J}$ with respect to the strong operator topology in $\mathcal{B}(E,G)$. Since $G \hookrightarrow \mathcal{S}'(\IR^{n};Y)$, it follows that $D_{m}f=\sum_{J \in \mathscr{J}}T_{m_{J}}\Delta_{J}f$ in $\mathcal{S}'(\IR^{n};Y)$.
On the other hand, $\widehat{T_{m}f} = m\hat{f} = \sum_{J \in \mathscr{J}}m_{J}\mathds{1}_{J}\hat{f}$ in $L^{1}(\IR^{n};Y)$, implying that $T_{m}f = \sum_{J \in \mathscr{J}}T_{m_{J}}\Delta_{J}f$ in $L^{\infty}(\IR^{n};Y) \hookrightarrow \mathcal{S}'(\IR^{n};Y)$. Therefore, $T_{m}f=D_{m}f$.
\end{proof}

\section{Littlewood-Paley Theory and Fourier Multipliers for $\mathcal{A}_{p}$-weights in One Dimension}\label{sec:LP1d}

In this section we extend Bourgain's Littlewood-Paley decomposition \cite{Bou86} to the weighted setting (see Theorem~\ref{thm:Littlewood-Paley_1d}),which we use to obtain a two-weight version of \cite[Theorem~3.4]{Wei} (see Theorem~\ref{thm:Mikhlin_1d}).

Although all results in this section are as special cases contained in Section~\ref{sec:rect_Ap-weights} on the higher-dimensional case, we have decided to treat the one-dimensional case separately. This has two reasons. Firstly, the one-dimensional case simplifies a lot and is already sufficient for the application to maximal $L^{p}$-regularity in Section~\ref{sec:application_max-Lp-reg}.
Secondly, this choice also improves the readability of Section~\ref{sec:rect_Ap-weights} at some points.

Throughout this section we will write $\mathcal{A}_{p}(\IR) := \mathcal{A}_{p}(\IR,\mathcal{R}_{1}) = \mathcal{A}_{p}(\IR,\mathcal{C}_{1})$.

\subsection{Littlewood-Paley Theory}\label{subsec:LP_1d}

\begin{lemma}\label{lem:cutoffs;1d}
		Let $X$ be a UMD space, $p \in (1, \infty)$ and $\omega, \sigma \in \mathcal{A}_p(\IR)$ with $[\omega, \sigma]_{\mathcal{A}_p(\IR)} < \infty$.
Then the family $\{ \Delta[I] : I \in \mathcal{R}_{1} \}$ lies in $\mathcal{B}(L^p_{\sigma}(\IR; X), L^p_{\omega}(\IR;X))$ with $\mathcal{R}$-bound
\[
\mathcal{R}\{\Delta_{j}[I]: I \in \mathcal{R}_{1} \}
					\lesssim_{X,p} [\omega, \sigma]_{\mathcal{A}_p(\IR)}^{1/p} ([\omega]_{\mathcal{A}_p(\IR)}^{1-\frac{1}{p}} + [\sigma]_{\mathcal{A}_p(\IR)}^{\frac{1}{p}}).
\]
	\end{lemma}
	\begin{proof}
		Since $X$ is UMD, the Hilbert transform $H$ defines a bounded operator on $L^p(\IR; X)$. The Hilbert transform is a Calderón--Zygmund operator, so $H$
is bounded between $L^p_{\sigma}(\IR; X)$ and $L^q_{\omega}(\IR;X)$ with a norm estimate as in Theorem~\ref{thm:extrapolation}. As a Fourier multiplier operator, $H$ has symbol $\imath \sgn = \imath (2\mathds{1}_{\IR_{+}}-\mathds{1})$ . This implies that $\mathds{1}_{\IR_{+}} \in \mathcal{M}^1_p((X,\sigma) \to (Y,\omega))$ with
\[
\normalnorm{\mathds{1}_{\IR_{+}}}_{\mathcal{M}^1_p((X,\sigma) \to (Y,\omega))} \lesssim_{X,p}
[\omega, \sigma]_{\mathcal{A}_p(\IR)}^{1/p} ([\omega]_{\mathcal{A}_p(\IR)}^{1-\frac{1}{p}} + [\sigma]_{\mathcal{A}_p(\IR)}^{\frac{1}{p}}).
\]
The proof can now be finished as in \cite[Lemma~3.7~c)]{KunWei04}.
	\end{proof}

\begin{lemma}\label{lem:Mikhlin_jth-coord;1d}
Let $X$ and $Y$ be UMD Banach spaces, $p \in (1,\infty)$ and $\omega, \sigma \in \mathcal{A}_p(\IR)$ with $[\omega,\sigma]_{\mathcal{A}_p(\IR)} < \infty$.
Then $\mathcal{S}(\IR) \subset \mathcal{M}^{1}_p((X,\sigma) \to (Y,\omega))$ with
\[
\normalnorm{\phi}_{\mathcal{M}^{1}_p((X,\sigma) \to (Y,\omega))} \lesssim [\omega, \sigma]_{\mathcal{A}_p(\IR)}^{1/p} ([\omega]_{\mathcal{A}_p(\IR)}^{1-\frac{1}{p}} + [\sigma]_{\mathcal{A}_p(\IR)}^{\frac{1}{p}})\sup_{k=0,\ldots,3}\sup_{\xi \in \IR}|\xi^{k}\phi^{(k)}(\xi)|
\]
for every $\phi \in \mathcal{S}(\IR)$.
\end{lemma}
\begin{proof}
Let $\phi \in \mathcal{S}(\IR)$ with $\sup_{k=0,\ldots,3}\sup_{\xi \in \IR}|\xi^{k}\phi^{(k)}(\xi)| \leq 1$.
The Mikhlin multiplier theorem (Theorem~\ref{thm:mikhlin}) in one dimension gives $\phi \in \mathcal{M}^1_{p}((X,\mathds{1}) \to (Y,\mathds{1}))$ with the estimate $\norm{\phi}_{\mathcal{M}^1_{p}((X,\mathds{1}) \to (Y,\mathds{1}))} \lesssim 1$.
By \cite[Proposition~VI.4.4.2(a)]{Ste93}, $K:=\mathcal{F}^{-1}\phi$ satisfies the estimates in the definition of a Calderón--Zygmund kernel independently of $\phi$.
The desired result thus follows from Theorem~\ref{thm:extrapolation}.
\end{proof}

\begin{lemma}\label{lem:density_cpct_Fourier_supp_away;1d}
Let $X$ be a Banach space, $p \in (1,\infty)$ and $\omega \in \mathcal{A}_p(\IR)$.
Then $\mathcal{F}^{-1}C_c^{\infty}(\IR_{*})$ is dense in $L^{p}_{\omega}(\IR;X)$, where $\IR_{*} = \IR \setminus \{0\}$.
\end{lemma}
\begin{proof}
In view of the density of $L^{p}_{\omega}(\IR) \otimes X$ in $L^{p}_{\omega}(\IR;X)$ we may without loss of generality assume that $X=\IC$. As $\mathcal{F}^{-1}C^{\infty}_{c}(\IR)$ is dense in $L^{p}_{\omega}(\IR)$, it suffices to show that $\mathcal{F}^{-1}C^{\infty}_{c}(\IR)$ is contained in the closure of $\mathcal{F}^{-1}C_c^{\infty}(\IR_{*})$ in $L^{p}_{\omega}(\IR)$.
So fix an $f \in \mathcal{F}^{-1}C^{\infty}_{c}(\IR)$.
For each $\varepsilon \in \{-1,1\}$ let $I_{\varepsilon} \coloneqq \varepsilon[0,\infty) \in \mathcal{R}_{1}$ and consider the associated frequency cut-off $\Delta(I_{\epsilon}) \in \mathcal{B}(L^{p}_{\omega}(\IR))$.
Then $f = \sum_{\varepsilon \in \{-1,1\}}\Delta(I_{\epsilon})f$ with $\supp \mathcal{F}[\Delta(I_{\epsilon})f] \subset I_{\varepsilon}$.
Furthermore, writing $e_{a}(x) = \exp(2\pi i a \cdot x)$, picking $\phi \in \mathcal{S}(\IR)$ with $\phi(0)=1$ and $\supp \mathcal{F}\phi \subset (0,\infty)$ and putting $\phi_{\varepsilon,k}:= \phi(\frac{\varepsilon}{k}\,\cdot\,)$, we have $f^{\varepsilon}_{k} := \phi_{\varepsilon,k}e_{\frac{1}{k}\varepsilon}\Delta(I_{\epsilon})f \in \mathcal{F}^{-1}C_c^{\infty}(\varepsilon(0,\infty))$ with $f^{\varepsilon}_{k} \to \Delta(I_{\epsilon})f$ in $L^{p}_{\omega}(\IR)$ as $k \to \infty$.
\end{proof}

For each $k \in \IZ$ and $\eta \in \{-1,1\}$ we consider the dyadic interval $I_{k,\eta} := \eta[2^{k},2^{k+1}]$.
Let $\mathcal{I}$ denote the collection of all these dyadic intervals: $\mathcal{I} \coloneqq \{ I_{k,\eta} : (k,\eta) \in \IZ \times \{-1,1\} \}$.

\begin{theorem}\label{thm:Littlewood-Paley_1d}
Let $X$ be a UMD Banach space, $p \in (1,\infty)$ and $\omega \in \mathcal{A}_p(\IR)$.
Then $\Delta \coloneqq (\Delta_{I})_{I \in \mathcal{I}}$ defines an unconditional decomposition of $L^{p}_{\omega}(\IR;X)$ with $C_{\Delta}^{\pm} \lesssim_{X,p} [\omega]_{\mathcal{A}_p(\IR)}^{2\max\{1,\frac{1}{p-1}\}}$.
\end{theorem}
\begin{proof}
Let us check the conditions of Proposition~\ref{prop:uncond_Schaud_bbd_into_Rad}.
The density of $\ran(\Delta) \supset \mathcal{F}^{-1}C_c^{\infty}(\IR_{*};X)$ in $L^{p}_{\omega}(\IR;X)$ follows from Lemma~\ref{lem:density_cpct_Fourier_supp_away;1d}.
For the randomized estimates we only need to treat $\Delta$, $\Delta^{*}$ being of the same form.
Indeed, as $X$ is reflexive (being a UMD space), $\Delta_{I}^{*} = \Delta_{-I}$ on $[L^{p}_{\omega}(\IR;X)]^{*} = L^{p'}_{\sigma}(\IR;X^{*})$, where $\sigma= \omega^{-1/(p-1)}$.
Furthermore, $[\omega]_{\mathcal{A}_p(\IR)}^{2\max\{1,\frac{1}{p-1}\}} = [\sigma]_{\mathcal{A}_{p'}(\IR)}^{2\max\{1,\frac{1}{p'-1}\}}$.

It is standard (and in fact only involving a direct computation) to construct $(\rho_I)_{I \in \mathcal{I}} \subset C^{\infty}_{c}(\IR)$ with the properties that (i)
$\rho_{I} \equiv 1$ on $I$ for each $I \in \mathcal{I}$ and that (ii) the functions
\[
\rho_{\varepsilon,\mathcal{J}}  \coloneqq \sum_{I \in \mathcal{J}}\varepsilon_{I}\rho_{I},
 \qquad \varepsilon \in \{-1,1\}^{\mathcal{I}}, \mathcal{J} \subset \mathcal{I}\:\text{finite},
\]
uniformly satisfy the Mikhlin condition of order $3$, that is, there exists a finite constant $C>0$ such that
\[
\sup\{ |\xi^{l}\rho_{\varepsilon,\mathcal{J}}^{(l)}(\xi)|  : l =0,\ldots,3, \xi \neq 0 \} \leq C
\]
for all $\varepsilon \in \{-1,1\}^{\mathcal{I}}$ and $\mathcal{J} \subset \mathcal{I}$ finite.
Using Lemma~\ref{lem:Mikhlin_jth-coord;1d} we find that $(\rho_{I})_{I \in \mathcal{I}} \subset \mathcal{M}^{1}_p(X,\omega)$ with
\[
\biggnorm{\sum_{I \in \mathcal{J}}\varepsilon_{I}T_{\rho_{I}}}_{\mathcal{B}(L^{p}_{\omega}(\IR;X))} \lesssim_{p,X} [\omega]_{\mathcal{A}_p(\IR)}^{\max\{1,\frac{1}{p-1}\}}
\]
for all  $\varepsilon \in \{-1,1\}^{\mathcal{I}}$ and $\mathcal{J} \subset \mathcal{I}$ finite.
As $\Delta_{I}T_{\rho_{I}} = \Delta_{I}$, combining this estimate with Lemma~\ref{lem:cutoffs;1d} gives the desired estimate for $\Delta$ in Proposition~\ref{prop:uncond_Schaud_bbd_into_Rad}.
\end{proof}

\subsection{The Mikhlin Fourier Multiplier Theorem}

The following theorem, which extends~\cite[Theorem~3.4]{Wei} to the two-weighted setting, is consequence of the generic Theorem~\ref{thm:generic_fm_thm} and the Littlewood--Paley decompositions from
Theorem~\ref{thm:Littlewood-Paley_1d}.

\begin{theorem}\label{thm:Mikhlin_1d}
Let $X$ and $Y$ be UMD Banach spaces, $p \in (1,\infty)$ and $\omega,\sigma \in \mathcal{A}_p(\IR)$ with $[\omega,\sigma]_{\mathcal{A}_p(\IR)} < \infty$.
\begin{thm_enum}%
			\item\label{item:thm:mikhlin_weighted;UMD;1d}
Let $m \in L^{\infty}(\IR;\mathcal{B}(X,Y))$ be $C^{1}$ on $\IR \setminus \{0\}$.
If
\begin{equation*}
\norm{m}_{\mathcal{R}\mathfrak{M}} := \sup_{k=0,1} \mathcal{R}\{ \xi^{k}m^{(k)}(\xi): \xi \neq 0\} < \infty,
\end{equation*}
then $m \in \mathcal{M}^{1}_p((X,\sigma),(Y,\omega))$ with
\begin{equation*}
\norm{T_m}_{\mathcal{B}(L^{p}_{\sigma}(\IR;X),L^{p}_{\omega}(\IR;Y))} \lesssim_{X,Y,p,\sigma,\omega} \norm{m}_{\mathcal{R}\mathfrak{M}}.
\end{equation*}
\item\label{item:thm:mikhlin_weighted;UMDalpha;1d} Suppose further that $X$ and $Y$ have Pisier's property~$(\alpha)$. If $\mathscr{M} \subset L^{\infty}(\IR^{n};\mathcal{B}(X,Y))$ is such that $\partial^{\alpha}m$ is $C^{1}$ on $\IR \setminus \{0\}$ for each $m \in \mathscr{M}$ and
    			\begin{equation*}
    				\norm{\mathscr{M}}_{\mathcal{R}\mathfrak{M}} := \sup_{k=0,1} \mathcal{R}\{\xi^{k} m^{(k)}(\xi): m \in \mathscr{M}, \xi \neq 0\} < \infty,
                \end{equation*}
                	then $\mathscr{M} \subset \mathcal{M}^{1}_p((X,\sigma),(Y,\omega))$ and one has in $\mathcal{B}(L^{p}_{\sigma}(\IR;X),L^{p}_{\omega}(\IR;Y))$
    			\begin{equation*}
					\mathcal{R}\{ T_{m} : m \in \mathscr{M} \} \lesssim_{X,Y,p,\sigma,\omega} \norm{\mathscr{M}}_{\mathcal{R}\mathfrak{M}}.
    			\end{equation*}
		\end{thm_enum}
\end{theorem}
\begin{proof}
By Theorems \ref{thm:generic_fm_thm} and \ref{thm:Littlewood-Paley_1d} we only need to check that $\{m\mathds{1}_{I} : I \in \mathcal{I}\}$ and $\{m\mathds{1}_{I} : m \in \mathscr{M}, I \in \mathcal{I}\}$ are of uniformly $\mathcal{R}$-bounded variation in \ref{item:thm:mikhlin_weighted;UMD;1d} and \ref{item:thm:mikhlin_weighted;UMDalpha;1d}, respectively.
The case \ref{item:thm:mikhlin_weighted;UMDalpha;1d} being exactly the same as \ref{item:thm:mikhlin_weighted;UMD;1d}, for simplicity of notation we only treat \ref{item:thm:mikhlin_weighted;UMD;1d}.

In connection with the representation \eqref{eq:def:uniform_R-bdd_var} in the definition of uniformly $\mathcal{R}$-bounded variation, let us note the following.
Let $-\infty < a < b < \infty$ and let $f:[a,b] \to Z$ be a $C^{1}$-function to some Banach space $Z$.
Then, extending $f$ by zero to $\IR$, the fundamental theorem of calculus gives
\begin{align*}
(f\mathds{1}_{\IR \setminus \{b\}})(\xi)
&= \int_{(-\infty,\xi]}f\mathds{1}_{\{a,b\}}\,d(\delta_{a}-\delta_{b}) + \int_{(-\infty,\xi]}f'\mathds{1}_{(a,b)}\,d\lambda.
\end{align*}
Denoting by $a_{I}$ and $b_{I}$ the left and right endpoint of $I \in \mathcal{I}$, respectively, this observation gives that, for a.e.\ $\xi \in \IR$,
\[
(m\mathds{1}_{I})(\xi) = \int_{(-\infty,\xi]}m\mathds{1}_{\{a_{I},b_{I}\}}\,d(\delta_{a_{I}}-\delta_{b_{I}}) + \int_{(-\infty,\xi]}\eta m'(\eta)\mathds{1}_{(a_{I},b_{I})}(\eta)\, \eta^{-1}d\lambda(\eta).
\]
So $m\mathds{1}_{I}$ satisfies \eqref{eq:def:uniform_R-bdd_var} a.e.\ with $\tau_{m,I}(\eta)=m\mathds{1}_{\{a_{I},b_{I}\}}+\eta m'(\eta)\mathds{1}_{(a_{I},b_{I})}(\eta)$ and $d\mu_{m,I}(\eta) = d(\delta_{a_{I}}-\delta_{b_{I}})+\mathds{1}_{(a_{I},b_{I})}(\eta)\eta^{-1}d\lambda(\eta)$.
Since
\[
\mathcal{R}(\{\tau_{m,I}(\eta):\eta \in \IR, I \in \mathcal{I}\}) \leq \norm{m}_{\mathcal{R}\mathfrak{M}}
\]
and
\begin{align*}
\normalnorm{\mu_{m,I}} = 2 + \int_{a_{I}}^{b_{I}}\frac{d\eta}{|\eta|} = 2 + \log(b_{I}/a_{I}) = 2 + \log(2), \qquad I = I_{k,\eta},
\end{align*}
it follows that
\[
\mathrm{var}_{\mathcal{R}}(\{ m\mathds{1}_{I}: I \in \mathcal{I} \}) \lesssim \norm{m}_{\mathcal{R}\mathfrak{M}}. \qedhere
\]
\end{proof}

\section{An application: maximal $L^{p}$-regularity}\label{sec:application_max-Lp-reg}

	We now give a short application of the obtained multiplier results in the context of maximal $L^p$-regularity. Let $-A$ be the generator of a bounded analytic $C_0$-semigroup on a Banach space $X$ (for an introduction see~\cite{Paz83} or~\cite{EngNag00}). Then $A$ is said to have \emph{maximal $L^p$-regularity} for $p \in (1, \infty)$ if for one or equivalently all $T \in (0, \infty)$ the following holds: for all $f \in L^p([0,T];X)$ the mild solution
	\begin{equation*}
		u(t) = \int_0^t e^{-(t-s)A} f(s) \d s
	\end{equation*}
	of the abstract Cauchy problem $\dot{u}(t) + Au(t) = f(t)$ with initial condition $u(0) = 0$ satisfies $u \in W^{1,p}([0,T];X) \cap L^p([0,T];D(A))$. By the closed graph theorem this is equivalent to the boundedness of the operator
	\begin{equation}
		\label{eq:mr_singular_operator}
		f \mapsto Au(\cdot) = \int_0^t Ae^{-(t-s)A} f(s) \d s = \int_{\IR} Ae^{-(t-s)A} \mathds{1}_{\IR_{\ge 0}}(t-s) f(s) \d s
	\end{equation}
	initially only defined for sufficiently regular functions, say $f \in C_c^{\infty}((0,T);X)$. Taking the Fourier transform, the boundedness of the singular integral at the right hand side is equivalent to the boundedness of the multiplier operator associated to $m(\xi) = i \xi (i\xi - A)^{-1}$. Since the $\mathcal{R}$-boundedness of $m$ is even a necessary condition for the boundedness of operator-valued multipliers and due to the easy structure of the resolvent, we even obtain -- partially as a consequence of the operator-valued Mikhlin multiplier theorem -- the following equivalence on UMD spaces:
	\begin{equation*}
		A \text{ has maximal } L^p\text{-regularity} \quad \Leftrightarrow \quad \R{i\xi(i\xi-A)^{-1}: \xi \neq 0} < \infty.
	\end{equation*}	
	For details we refer to the first chapters of~\cite{KunWei04}. This is the celebrated characterization of maximal $L^p$-regularity on UMD spaces due to Weis~\cite{Wei}. Using our weighted Mikhlin multiplier result (Theorem~\ref{thm:Mikhlin_1d}), we obtain the following corollary.
	
	\begin{corollary}\label{cor:mr}
		Let $-A$ be the generator of a bounded analytic semigroup on some UMD space $X$. Suppose further that $A$ has maximal $L^p$-regularity for some $p \in (1, \infty)$. Then for all $p \in (1, \infty)$ and $\omega \in \mathcal{A}_p(\IR;\mathcal{Q}_1)$ one has maximal $L^p$-regularity in the following sense: for all $f \in L^p_{\omega}([0,T];X)$ the abstract Cauchy problem
    	\begin{equation*}
    		\left\{
    		\begin{aligned}
    			\dot{u}(t) + Au(t) & = f(t) \\
    			u(0) & = 0
    		\end{aligned}
    		\right.
    	\end{equation*}
		has a unique solution $u$ in $W^{1,p}_{\omega}([0,T];X) \cap L^p_{\omega}([0,T];X)$.
	\end{corollary}
	
	Here $W^{1,p}_{\omega}([0,T];X)$ is the space of all $X$-valued distributions for which both $u$ and $\dot{u}$ lie in $L^p_{\omega}([0,T];X)$. Note that Theorem~\ref{thm:mikhlin_weighted} actually gives a two-weight result for the operator~\eqref{eq:mr_singular_operator}. %
	
	Corollary~\ref{cor:mr} was first shown in~\cite{PruSim04} for power weights in $\mathcal{A}_p$ with positive exponents and was subsequently generalized to all $\mathcal{A}_p$ power weights in~\cite[Theorem~1.15]{HaaKun07}. For general $\mathcal{A}_p$-weights the result was first shown in~\cite[Corollary~5]{ChiFio14} as a consequence of the extrapolation result for Calderón--Zygmund operators (Theorem~\ref{thm:extrapolation}). However, one now sees that this extrapolation result for maximal $L^p$-regularity follows automatically from the extrapolation properties of Mikhlin multipliers and therefore is inherent to the standard approach via $\mathcal{R}$-boundedness estimates on the resolvent.
	
	We finally remark that the result of Corollary~\ref{cor:mr} actually holds for a broader class of weights than the $\mathcal{A}_p$-weights. In fact, since the kernel vanishes on the negative real line, maximal $L^p$-regularity even holds for $\omega \in \mathcal{A}_{p}^{-}$, a class of one-sided Muckenhoupt weights~\cite[Theorem~5.1]{ChiKro14}.

\section{Littlewood-Paley Theory and Fourier Multipliers for Rectangular $\mathcal{A}_{p}$-weights}\label{sec:rect_Ap-weights}\label{sec:LP}

\subsection{Product Pre-Decompositions and Blockings}

\begin{theorem}\label{thm:unconditional_product_decompositions}
Let $X$ be a Banach space with the property that both $X$ and $X^{*}$ have Pisier's property~$(\alpha)$ and let $\Delta^{j} = (\Delta^{j}_{i_{j}})_{i_{j} \in I_{j}}$, $j=1,\ldots,n$, be commuting unconditional decompositions of $X$.
Put $I:=\prod_{j=1}^{n}I_{j}$ and, for each $i=(i_{1},\ldots,i_{n}) \in I$, $\Delta_{i} := \prod_{j=1}^{n}\Delta^{j}_{i_{j}}$.
Then $\Delta = (\Delta_{i})_{i \in I}$ is an unconditional decomposition of $X$ with
\[
C^{+}_{\Delta} \leq \alpha_{X}^{n-1}C^{+}_{\Delta^{1}}\ldots C^{+}_{\Delta^{n}} \quad \text{and} \quad C^{-}_{\Delta} \leq C^{+}_{\Delta^{*}} \leq \alpha_{X^{*}}^{n-1}C^{+}_{(\Delta^{1})^{*}}\ldots C^{+}_{(\Delta^{n})^{*}}.
\]
\end{theorem}
\begin{proof}
Although this result seems to be well known, we do not know an explicit reference. However, the argumentation used in the concrete setting of Littlewood-Paley decompositions (see for instance \cite[Proposition~4.12]{KunWei04} and the corresponding note \cite[N.~4.12]{KunWei04}) also works in our setting. The argument goes as follows. Since one readily sees that $\mathrm{ran}(\Delta)$ is dense in $X$, by Lemma~\ref{lem:basic_char_uncond_Schaud_decomp} it suffices to show that both $\Delta$ and $\Delta^{*}$ are U$^{+}$ (with U$^{+}$-constants as asserted), something which follows directly from \cite[Proposition~7.5.4]{HNVW17}.
\end{proof}

\begin{remark}\label{rmk:thm:unconditional_product_decompositions;K-convex}
For a K-convex Banach space $X$ it holds that $X$ has Pisier's property~$(\alpha)$ if and only if $X^{*}$ does (see \cite[Proposition~7.5.15]{HNVW17}).
Moreover, $\alpha_{X} \leq K^{2}_{2,X}\alpha_{X^{*}}$ and $\alpha_{X^{*}} \leq K^{2}_{2,X^{*}}\alpha_{X}$.
In connection to the Littlewood--Paley theory in the next subsection, let us mention that every UMD space $X$ is K-convex with $K_{p,X} \leq \beta^{+}_{p,X} \leq \beta_{p,X}$ for all $p \in (1,\infty)$ (see \cite[Proposition~4.3.10]{HNVW16}).
\end{remark}

In the absence of property $(\alpha)$ the product pre-decomposition above is in general not unconditional.
In fact, in the context of Littlewood--Paley decompositions it even occurs that property $(\alpha)$ is not only sufficient but also necessary, see \cite{Lancien1998}.
However, as the next theorem shows, under some $\mathcal{R}$-boundedness conditions, one can find an appropriate blocking of the product pre-decomposition which forms an unconditional decomposition. The theorem is a modification of \cite[Theorem~2.5.1]{Wit00}, which was inspired by the work~\cite{Zim89} on multi-dimensional Littlewood-Paley decompositions.

Before we state the theorem, let us introduce some notation. Given an unconditional decomposition $\Delta=(\Delta_{i})_{i \in I}$ of $X$ and a subset $J \subset I$, we define in the strong operator topology
\[
\Delta_{J} := \mathrm{SOT}-\sum_{i \in J}\Delta_{i}.
\]

\begin{theorem}\label{thm:unconditional_blockings_product_decompositions}
Let $\Delta^{j} = (\Delta^{j}_{i})_{i \in \IZ}$, $j=1,\ldots,n$, be commuting unconditional decompositions of a Banach space $X$.
Suppose that the following $\mathcal{R}$-boundedness conditions hold true for all $j = 1, \ldots, n$:
\begin{equation}\label{eq:thm:unconditional_blockings_product_decompositions;R-bdd_assumption}
\kappa_{j} := \mathcal{R}\left\{ \sum_{i=M}^{N}\Delta^{j}_{i} \: : \: M,N \in \IZ \right\} < \infty, \quad
\kappa^{*}_{j} := \mathcal{R}\left\{ \sum_{i=M}^{N}(\Delta^{j}_{i})^{*} \: : \: M,N \in \IZ \right\} < \infty.
\end{equation}
Define the partition $(J_{k})_{k \in \IZ}$ of the index set $\IZ^{n}$ by
\[
J_{ln+r} := (\IZ \cap (-\infty, l+1])^{r} \times \{l+1\} \times (\IZ \cap (-\infty, l])^{n-r-1},
\]
where $l \in \IZ$ and $r \in \{0,\ldots,n-1\}$.
For each $k \in \IZ$ we define the bounded linear projection
\[
\Delta_{k} := \mathrm{SOT}-\sum_{i \in J_{k}}\Delta^{1}_{i_{1}} \cdots \Delta^{n}_{i_{n}}
\]
in $X$.
Then $\Delta = (\Delta_{k})_{k \in \IZ}$ is an unconditional decomposition of $X$ for which we have
\begin{equation}\label{fm:eq:thm;unconditional_blockings_product_decompositions;U-plus_constants}
C^{+}_{\Delta} \leq \sum_{j=1}^{n}C^{+}_{\Delta^{j}} \prod_{j \neq k}\kappa_{k}, \quad\quad
C^{-}_{\Delta} \leq C^{+}_{\Delta^{*}} \leq \sum_{j=1}^{n}C^{+}_{(\Delta^{j})^{*}} \prod_{j \neq k}\kappa_{k}^{*}
\end{equation}
and
\begin{equation}\label{fm:eq:thm;unconditional_blockings_product_decompositions;R-bdd_conclusion}
\mathcal{R}\left\{ \sum_{k=M}^{N}\Delta_{k} \: : \: N,M \in \IZ \right\} \leq 2\kappa_{1}\cdots\kappa_{n}, \quad\quad
\mathcal{R}\left\{ \sum_{k=M}^{N}\Delta_{k}^{*} \: : \: N,M \in \IZ \right\} \leq 2\kappa_{1}^{*}\cdots\kappa_{n}^{*}.
\end{equation}
\end{theorem}

\begin{remark}\label{rmk:thm:unconditional_blockings_product_decompositions}
Concerning the $\mathcal{R}$-boundedness assumptions in Theorem~\ref{thm:unconditional_blockings_product_decompositions}, let us remark the following. The $\mathcal{R}$-boundedness of the first collections in
\eqref{eq:thm:unconditional_blockings_product_decompositions;R-bdd_assumption}
is automatic when the space $X$ has the so-called triangular contraction property (or property weak-$(\alpha)$); see \cite[Definition~2.4.1]{Wit00} and \cite[Corollary~2.4.3]{Wit00}.
Having the $\mathcal{R}$-boundedness of the first collections,
the $\mathcal{R}$-boundedness of the second collections then is a consequence for $K$-convex spaces; see e.g.\ \cite[Proposition~8.20]{HNVW17}. In particular, the $\mathcal{R}$-boundedness assumption~\eqref{eq:thm:unconditional_blockings_product_decompositions;R-bdd_assumption}
is automatic when $X$ is a UMD space; see \cite{HNVW16}.
\end{remark}

In the next section we will apply Theorems \ref{thm:unconditional_product_decompositions} and \ref{thm:unconditional_blockings_product_decompositions} in the setting of Littlewood--Paley decompositions.
There the $\mathcal{R}$-bounds in \eqref{eq:thm:unconditional_blockings_product_decompositions;R-bdd_assumption} can be checked directly, with explicit bounds, so that we do not have to rely on the above remark.

\begin{proof}[Proof of Theorem~\ref{thm:unconditional_blockings_product_decompositions}.]
For simplicity of notation we only treat the case $n=2$. Throughout the proof it will also be convenient to write $P^{j}_{k} := \Delta^{j}_{\IZ \cap (-\infty,k]}$ for each $j \in \{1,2\}$ and $k \in \IZ$.
From \eqref{eq:thm:unconditional_blockings_product_decompositions;R-bdd_assumption} and the preservation of $\mathcal{R}$-bounds under taking closures in $\mathcal{B}(X)$ and $\mathcal{B}(X^{*})$ with respect to the $\mathrm{SOT}$-topology and the $\mathrm{W}^{*}\mathrm{OT}$-topology, respectively, it follows that
\begin{equation}\label{fm:eq:thm;unconditional_blockings_product_decompositions;R-bdd_conclusion;P}
\mathcal{R}\{ P^{j}_{k} \: : \: k \in \IZ \} \leq \kappa_{j} \quad\quad \mbox{and} \quad
\mathcal{R}\{ (P^{j}_{k})^{*} \: : \: k \in \IZ \} \leq \kappa_{j}^{*}, \quad\quad j=1,2.
\end{equation}

One readily sees that $\mathrm{ran}(\Delta)$ is dense in $X$.
In view of Proposition~\ref{prop:uncond_Schaud_bbd_into_Rad}, in order to show that $\Delta$ is an unconditional decomposition with \eqref{fm:eq:thm;unconditional_blockings_product_decompositions;U-plus_constants} it thus suffices that both $\Delta$ and $\Delta^{*}$ are \emph{U$^{+}$}, with $C^{+}_{\Delta} \leq C^{+}_{\Delta^{1}}\kappa_{2} + C^{+}_{\Delta^{2}}\kappa_{1}$,
$C^{+}_{\Delta^{*}} \leq C^{+}_{(\Delta^{1})^{*}}\kappa_{2}^{*} + C^{+}_{(\Delta^{2})^{*}}\kappa_{1}^{*}$.
We only consider $\Delta$, the case of $\Delta^{*}$ being completely similar.
To this end, let $x \in \mathrm{ran}(\Delta)$ and a finite subset $F$ of $\IZ$ be given.
Writing $F = F_{0} \cup F_{1}$ with $F_{r}:= F \cap [2\IZ + r]$ for $r \in \{0,1\}$,
it suffices to show that
\[
\biggnorm{\sum_{n \in F_{0}}\epsilon_{n}\Delta_{n}x}_{L^{2}(\Omega;X)} \leq C^{+}_{\Delta^{2}}\kappa_{1}, \quad\quad
\biggnorm{\sum_{n \in F_{1}}\epsilon_{n}\Delta_{n}x}_{L^{2}(\Omega;X)} \leq C^{+}_{\Delta^{1}}\kappa_{2}.
\]
We only treat the random sum over $F_{1}$, the sum over $F_{0}$ being similar.
Defining $\tilde{F}_{1} := \{ l \in \IZ : 2l+1 \in F_{1} \}$ and using $\Delta_{2l+1} = \Delta^{1}_{l+1}P^{2}_{l+1} = P^{2}_{l+1}\Delta^{1}_{l+1}$ and $x \in \mathrm{Ran}(\Delta) \subset \overline{\mathrm{Ran}(\Delta^{1})}$,
we find
\begin{align*}
\biggnorm{\sum_{n \in F_{1}}\epsilon_{n}\Delta_{n}x}_{L^{2}(\Omega;X)}
&= \biggnorm{\sum_{l \in \tilde{F}_{1}}\epsilon_{2l+1}P^{2}_{l+1}\Delta^{1}_{l+1}x}_{L^{2}(\Omega;X)}
 \stackrel{\eqref{fm:eq:thm;unconditional_blockings_product_decompositions;R-bdd_conclusion;P}}{\leq} \kappa_{2} \biggnorm{\sum_{l \in \tilde{F}_{1}}\epsilon_{2l+1}\Delta^{1}_{l+1}x}_{L^{2}(\Omega;X)} \\
&\leq \kappa_{2}C^{+}_{\Delta^{1}}\norm{x}_{X}. \qedhere
\end{align*}
\end{proof}

Let us finally derive the $\mathcal{R}$-bounds in \eqref{fm:eq:thm;unconditional_blockings_product_decompositions;R-bdd_conclusion}.
Define $(\Pi_{k})_{k \in \IZ}$ by $\Pi_{k} := \Delta_{\IZ \cap (-\infty,k]}$. Then, on the one hand we have $\sum_{k=M}^{N}\Delta_{k} = \Pi_{N}-\Pi_{M-1}$ for $N \geq M$ and $\sum_{n=M}^{N}\Delta_{n}=0$ otherwise.
On the other hand,
\[
\Pi_{k} = \left\{\begin{array}{ll}
P^{1}_{l+1}P^{2}_{l+1} & k=2l+1, l \in \IZ,\\
P^{1}_{l+1}P^{2}_{l} & k=2l, l \in \IZ,
\end{array}\right.
\]
so that $(\Pi_{k} )_{ k \in \IZ } \subset \{P^{1}_{k} \: : \: k \in \IZ \} \,\cdot\, \{ P^{2}_{k} \: : \: k \in \IZ \}$ and thus $(\Pi_{k}^{*} )_{ k \in \IZ } \subset \{(P^{2}_{k})^{*} \: : \: k \in \IZ \} \,\cdot\, \{ (P^{1}_{k})^{*} \: : \: k \in \IZ \}$.
The $\mathcal{R}$-bounds in \eqref{fm:eq:thm;unconditional_blockings_product_decompositions;R-bdd_conclusion} thus follow from \eqref{fm:eq:thm;unconditional_blockings_product_decompositions;R-bdd_conclusion;P}.

\subsection{Littlewood-Paley decompositions}

In this subsection we prove Littlewood--Paley decompositions in the vector-valued weighted setting.
More specifically, the aim is to obtain Theorem~\ref{thm:littlewood-paley}. As already mentioned in the introduction of this paper, $\mathcal{A}_p(\IR^n;\mathcal{R}_n)$ is the right class of weights for doing such Littlewood--Paley decompositions. As a matter of fact, the Littlewood--Paley decompositions require
\[
\mathds{1}_{\{x_{1} \geq 1\}},\ldots,\mathds{1}_{\{x_{n} \geq 1\}} \in \mathcal{M}^n_{p}(X,\omega)
\]
while it is known from \cite{Kur80} that
\[
\mathds{1}_{\{x_{1} \geq 1\}},\ldots,\mathds{1}_{\{x_{n} \geq 1\}} \in \mathcal{M}^n_{p}(\IC,\omega) \Longleftrightarrow  \omega \in \mathcal{A}_p(\IR^n;\mathcal{R}_n).
\]
The following lemma describes the one-dimensional behaviour of the class $\mathcal{A}_p(\IR^n;\mathcal{R}_n)$ in the two-weight setting.

	\begin{lemma}\label{lem:relation_ap}
		Let $p \in (1, \infty)$ and $\omega, \sigma\colon \IR^n \to \IR_{\ge 0}$ weights with $[\omega,\sigma]_{\mathcal{A}_p(\IR^n;\mathcal{R}_n)} < \infty$. Then for all $j \in \{1, \ldots, n \}$ and almost every $(x_1, \ldots, x_{j-1}, x_{j+1}, \ldots, x_n)$ one has
		\begin{align*}
			\MoveEqLeft \relax [\omega(x_1, \ldots, x_{j-1}, \cdot, x_{j+1}, \ldots, x_n), \sigma(x_1, \ldots, x_{j-1}, \cdot, x_{j+1}, \ldots, x_n)]_{\mathcal{A}_p(\IR; \mathcal{R}_1)} \\
			& \le [\omega,\sigma]_{\mathcal{A}_p(\IR^n;\mathcal{R}_n)}.
		\end{align*}
		In particular, if $\omega \in \mathcal{A}_p(\IR^n;\mathcal{R}_n)$, then
		\begin{equation*}
			[\omega(x_1, \ldots, x_{j-1}, \cdot, x_{j+1}, \ldots, x_n)]_{\mathcal{A}_p(\IR; \mathcal{R}_1)} \le [\omega]_{\mathcal{A}_p(\IR^n; \mathcal{R}_n)}
		\end{equation*}
		for all $j \in \{1, \ldots, n \}$ and almost every $(x_1, \ldots, x_{j-1}, x_{j+1}, \ldots, x_n)$.
	\end{lemma}
	\begin{proof}
The proof that we present is a direct adaption of the one-weighted argument in~\cite[p.~241]{Kur80}. Suppose that $[\omega,\sigma]_{\mathcal{A}_p(\IR^n;\mathcal{R}_n)} < \infty$. We may assume that $j = 1$. Let $I \subset \IR$ be an interval and $Q \subset \IR^{n-1}$ a cube, both of positive and finite measure. Then
		\begin{align*}
			\MoveEqLeft \biggl( \frac{1}{\abs{Q}} \int_Q \frac{1}{\abs{I}} \int_I \omega(y,x) \d y \d x \biggr) \biggl( \frac{1}{\abs{Q}} \int_Q \frac{1}{\abs{I}} \int_I \sigma(y,x)^{-\frac{1}{p-1}} \d y \d x \biggr)^{p-1} \\
			& = \biggl( \frac{1}{\abs{Q \times I}} \int_{Q \times I} \omega(y,x) \d y \d x \biggr) \biggl( \frac{1}{\abs{Q \times I}} \int_{Q \times I} \sigma(y,x)^{-\frac{1}{p-1}} \d y \d x \biggr)^{p-1} \\
			& \le [\omega, \sigma]_{\mathcal{A}_p(\IR^n; \mathcal{R}_n)}.
		\end{align*}
		Now, for fixed $x = (x_2, \ldots, x_n) \in \IR^{n-1}$ choose cubes centered at this point and shrinking to volume zero. For a fixed $I$ the desired estimate follows for almost every $(x_2, \ldots, x_n)$ from Lebesgue's differentiation theorem.
		A universal exceptional set independent of $I$ can be found by first considering only intervals with rational endpoints and then passing to general ones with a limiting argument.
	\end{proof}

For establishing the Littlewood--Paley decompositions of Theorem~\ref{thm:littlewood-paley}, together with Theorems \ref{thm:unconditional_product_decompositions} and \ref{thm:unconditional_blockings_product_decompositions}, the above lemma basically allows us to reduce the problem to the one-dimensional case (in the form of Lemma~\ref{lem:Littlewood-Paley_coordinate}), which was already treated in Section~\ref{subsec:LP_1d}.
This reduction requires some (notational) preparations in our setting.

Given $j \in \{1,\ldots,n\}$ and $T\colon L^{p}(\IR;X) \to L^{p}(\IR;Y)$, we let $T_{j}\colon L^{p}(\IR^{n};X) \to L^{p}(\IR^{n};Y)$ be the pointwise well-defined induced operator
\[
(T_{j}f)(x) = (Tf(x_{1},\ldots,x_{j-1},\,\cdot\,,x_{j+1},\ldots,x_{n}))(x_{j}).
\]
In this notation, the above lemma combined with Theorem~\ref{thm:extrapolation} immediately yields:

\begin{lemma}\label{lem:weighted_est_CZ-op_fix_var}
Let $T\colon L^{p}(\IR;X) \to L^{p}(\IR;Y)$ be a Calderón--Zygmund operator (as defined in Section~\ref{sec:extrapolation_CZ}) for some given $p \in (1,\infty)$.
For all $\omega, \sigma \in \mathcal{A}_p(\IR^n;\mathcal{R}_n)$ with $[\omega,\sigma]_{\mathcal{A}_p(\IR^n;\mathcal{R}_n)} < \infty$ and $f \in L^{p}(\IR^{n};X) \cap L^{p}_{\sigma}(\IR^{n};X)$ there holds the estimate
\begin{equation}\label{eq:weighted_estimate_one_coordinate}
\normalnorm{T_{j}f}_{L^{p}_{\omega}(\IR^{n};Y)} \lesssim [\omega, \sigma]_{\mathcal{A}_p(\IR^n; \mathcal{R}_n)}^{1/p} ([\omega]_{\mathcal{A}_p(\IR^n;\mathcal{R}_n)}^{1-\frac{1}{p}} + [\sigma]_{\mathcal{A}_p(\IR^n;\mathcal{R}_n)}^{\frac{1}{p}}) \norm{f}_{L^{p}_{\sigma}(\IR^{n};X)}.
\end{equation}
The implicit constant only depends on $\norm{T}_{\mathcal{B}(L^p(\IR;X),L^p(\IR;Y))}$, $p$ and the constant $C$ in the definition of a Calderón--Zygmund kernel.
\end{lemma}

It will be convenient to introduce the following notation. For each $j \in \{1,\ldots,n\}$ we define $\pi_{j}\colon \IR^{n} \to \IR$ by $\pi_{j}x:= x_{j}$ and consider the associated pull-back on functions: for a function $f\colon \IR \to \IC$ we write $\pi_{j}^{*}f \coloneqq f \circ \pi_{j}$.
Let $m \in L^{\infty}(\IR;\mathcal{B}(X,Y))$ be such that $m \in \mathcal{M}_{p}^1((X,\mathds{1}) \to (Y,\mathds{1}))$.
Then observe that $\pi_{j}^{*}m \in \mathcal{M}_{p}^n((X,\mathds{1}), (Y,\mathds{1}))$ with
\begin{equation}\label{eq:identity_fm_fixed_var}
[T_{m}]_{j} = T_{\pi_{j}^{*}m} \quad \text{in}\quad \mathcal{B}(L^{p}(\IR^{n};X),L^{p}(\IR^{n};Y)).
\end{equation}
For $j \in \{1,\ldots,n\}$ and a measurable set $A \subset \IR$ we define the frequency cut-off with respect to the $j$-coordinate $\Delta_{j}[A]$ by $\Delta_{j}[A] := T_{\pi_{j}^{*}\mathds{1}_{A}}$.

	\begin{lemma}\label{lem:cutoffs}
		Let $X$ be a UMD space, $n \in \IN$, $p \in (1, \infty)$ and $\omega, \sigma \in \mathcal{A}_p(\IR^n; \mathcal{R}_n)$ with $[\omega, \sigma]_{\mathcal{A}_p(\IR^n; \mathcal{R}_n)} < \infty$. For each $j \in \{1,\ldots,n\}$ the family of spectral projections $\{ \Delta_{j}[I] : I \in \mathcal{R}_{1} \}$ lies in $\mathcal{B}(L^p_{\sigma}(\IR^n; X), L^p_{\omega}(\IR^n;X))$ with $\mathcal{R}$-bound
\begin{equation}\label{eq:lem:cutoffs}
\mathcal{R}\{\Delta_{j}[I]: I \in \mathcal{R}_{1} \}
					\lesssim_{X,p} [\omega, \sigma]_{\mathcal{A}_p(\IR^n; \mathcal{R}_n)}^{1/p} ([\omega]_{\mathcal{A}_p(\IR^n;\mathcal{R}_n)}^{1-\frac{1}{p}} + [\sigma]_{\mathcal{A}_p(\IR^n;\mathcal{R}_n)}^{\frac{1}{p}}).
\end{equation}
As a consequence, $\{ \Delta[R]: R \in \mathcal{R}_{n} \} \subset \mathcal{B}(L^p_{\sigma}(\IR^n; X), L^p_{\omega}(\IR^n;X))$ with $\mathcal{R}$-bound
\[
\mathcal{R}\{\Delta[R]: R \in \mathcal{R}_{n} \} \lesssim_{X,p} \left( [\omega, \sigma]_{\mathcal{A}_p(\IR^n; \mathcal{R}_n)}^{1/p} ([\omega]_{\mathcal{A}_p(\IR^n;\mathcal{R}_n)}^{1-\frac{1}{p}} + [\sigma]_{\mathcal{A}_p(\IR^n;\mathcal{R}_n)}^{\frac{1}{p}}) \right)^{n}.
\]
	\end{lemma}
	\begin{proof}
We only need to prove the first statement, including the $\mathcal{R}$-bound \eqref{eq:lem:cutoffs}. This can can be done in the same way as Lemma~\ref{lem:cutoffs;1d},
now using Lemma~\ref{lem:weighted_est_CZ-op_fix_var} in combination with the simple observation \eqref{eq:identity_fm_fixed_var} instead of directly using Theorem~\ref{thm:extrapolation}.
	\end{proof}

The following lemma can be obtained in the same way as Lemma~\ref{lem:Mikhlin_jth-coord;1d}, now using Lemma~\ref{lem:weighted_est_CZ-op_fix_var} in combination with the simple observation \eqref{eq:identity_fm_fixed_var} instead of directly using Theorem~\ref{thm:extrapolation}.
\begin{lemma}\label{lem:Mikhlin_jth-coord}
Let $X$ and $Y$ be UMD Banach spaces, $p \in (1,\infty)$ and $\omega, \sigma \in \mathcal{A}_p(\IR^n;\mathcal{R}_n)$ with $[\omega,\sigma]_{\mathcal{A}_p(\IR^n;\mathcal{R}_n)} < \infty$.
For every $j \in \{1,\ldots,n\}$ and $\phi \in \mathcal{S}(\IR)$ it holds that
$\pi_{j}^{*}\phi \in \mathcal{M}^n_p((X,\sigma) \to (Y,\omega))$ with
\[
\normalnorm{\pi_{j}^{*}\phi}_{\mathcal{M}^n_p((X,\sigma) \to (Y,\omega))} \lesssim [\omega, \sigma]_{\mathcal{A}_p(\IR^n; \mathcal{R}_n)}^{1/p} ([\omega]_{\mathcal{A}_p(\IR^n;\mathcal{R}_n)}^{1-\frac{1}{p}} + [\sigma]_{\mathcal{A}_p(\IR^n;\mathcal{R}_n)}^{\frac{1}{p}})\sup_{k=0,\ldots,3}\sup_{\xi \in \IR}|\xi^{k}\phi^{(k)}(\xi)|.
\]
\end{lemma}

\begin{lemma}\label{lem:density_cpct_Fourier_supp_away}
Let $X$ be a Banach space, $p \in (1,\infty)$ and $\omega \in \mathcal{A}_p(\IR^n;\mathcal{R}_n)$.
Then $\mathcal{F}^{-1}C_c^{\infty}(\IR^{n}_{*};X)$
is dense in $L^{p}_{\omega}(\IR^{n};X)$, where $\IR^{n}_{*} = [\IR \setminus \{0\}]^{n}$.
\end{lemma}
\begin{proof}
This can be proved in the same way as Lemma~\ref{lem:density_cpct_Fourier_supp_away;1d}, now using $R_{\epsilon} \coloneqq \prod_{j=1}^{n}\varepsilon_{j}[0,\infty) \in \mathcal{R}_{n}$ with $\varepsilon \in \{-1,1\}^{n}$ instead of $I_{\varepsilon} \in \mathcal{R}_{1}$ with $\varepsilon \in \{-1,1\}$.
Furthermore, one has to take $\phi_{\varepsilon,k}(x):=\phi(\frac{\varepsilon_{1}}{k}x_{1},\ldots,\frac{\varepsilon_{n}}{k}x_{n})$.
\end{proof}

Recall from Section~\ref{sec:LP1d} the collection of dyadic interval $\mathcal{I} = \{ I_{k,\eta} : (k,\eta) \in \IZ \times \{-1,1\} \}$, where $I_{k,\eta} = \eta[2^{k},2^{k+1}]$.

\begin{lemma}\label{lem:Littlewood-Paley_coordinate}
Let $X$ be a UMD Banach space, $p \in (1,\infty)$ and $\omega \in \mathcal{A}_p(\IR^n;\mathcal{R}_n)$.
For each $j \in \{1,\ldots,n\}$, $\Delta_{j} \coloneqq ((\Delta_{I})_j)_{I \in \mathcal{I}}$ defines an unconditional decomposition of $L^{p}_{\omega}(\IR^{n};X)$ with $C_{\Delta_{j}}^{\pm} \lesssim_{X,p} [\omega]_{\mathcal{A}_p(\IR^n;\mathcal{R}_n)}^{2\max\{1,\frac{1}{p-1}\}}$.
\end{lemma}
\begin{proof}
Let us check the conditions of Proposition~\ref{prop:uncond_Schaud_bbd_into_Rad}.
The density of $\ran(\Delta_{j}) \supset L^{p}_{\omega}(\IR^{n};X) \cap \mathcal{F}^{-1}C_c^{\infty}(\IR^{n}_{*};X)$ in $L^{p}_{\omega}(\IR^{n};X)$ follows from Lemma~\ref{lem:density_cpct_Fourier_supp_away}.
For the randomized estimates we only need to treat $\Delta_{j}$, $(\Delta_{j})^{*}$ being of the same form.
Indeed, as $X$ is reflexive (being a UMD space), $(\Delta_{I,j})^{*} = \Delta_{-I,j}$ on $[L^{p}_{\omega}(\IR^{n};X)]^{*} = L^{p'}_{\sigma}(\IR^{n};X^{*})$, where $\sigma= \omega^{-1/(p-1)}$.
Furthermore, $[\omega]_{\mathcal{A}_p(\IR^n;\mathcal{R}_n)}^{2\max\{1,\frac{1}{p-1}\}} = [\sigma]_{\mathcal{A}_{p'}(\IR^n;\mathcal{R}_n)}^{2\max\{1,\frac{1}{p'-1}\}}$.

Let $(\rho_I)_{I \in \mathcal{I}} \subset C^{\infty}_{c}(\IR)$ be as in the proof of Theorem~\ref{thm:Littlewood-Paley_1d}.
Using Lemma~\ref{lem:Mikhlin_jth-coord} we find that $(\pi_{j}^{*}\rho_{I})_{I \in \mathcal{I}} \subset \mathcal{M}^n_p(X,\omega)$ with
\[
\biggnorm{\sum_{I \in \mathcal{J}}\varepsilon_{I}T_{\pi_{j}^{*}\rho_{I}}}_{\mathcal{B}(L^{p}_{\omega}(\IR^{n};X))} \lesssim_{p,X} [\omega]_{\mathcal{A}_p(\IR^n;\mathcal{R}_n)}^{\max\{1,\frac{1}{p-1}\}}
\]
for all  $\varepsilon \in \{-1,1\}^{\mathcal{I}}$ and $\mathcal{J} \subset \mathcal{I}$ finite.
As $\Delta_{I,j}T_{\pi_{j}^{*}\rho_{I}} = \Delta_{I,j}$, combining this estimate with Lemma~\ref{lem:cutoffs} gives the desired estimate for $\Delta_{j}$ in Proposition~\ref{prop:uncond_Schaud_bbd_into_Rad}.
\end{proof}

We are now able to prove the Littlewood--Paley decompositions that we will use to obtain the Mikhlin multiplier theorems in Section~\ref{subsec:mf_Mikhlin} via an application of the abstract multiplier result Theorem~\ref{thm:mikhlin_weighted}.
For this we apply Theorems \ref{thm:unconditional_product_decompositions} and \ref{thm:unconditional_blockings_product_decompositions} to the above unconditional decompositions.
In the presence of Pisier's property $(\alpha)$ we can use Theorem~\ref{thm:unconditional_product_decompositions} and simply take the product decomposition, which consists of the spectral projections corresponding to rectangles from the family $\mathcal{I}_{n} := \{I_{1} \times \ldots \times I_{n}: I_{1},\ldots,I_{n} \in \mathcal{I}\}$.
In the general case only Theorem~\ref{thm:unconditional_blockings_product_decompositions} on blockings is applicable, which leads us to consider the family of rectangles $\mathcal{E}_{n} = \{ E_{k,\eta} : (k,\eta) \in \IZ \times \{-1,1\}^{n} \}$ defined for $l \in \IZ$ and $r \in \{0, \ldots, n-1\}$ by
\[
E_{ln + r,\eta} \coloneqq \prod_{j=1}^{r}\eta_{j}[0,2^{l+1}] \times \eta_{r}[2^{l},2^{l+1}] \times \prod_{j=r+2}^{n}\eta_{j}[0,2^{l}],
\]
Note that for $J_{k}$ is as in Theorem~\ref{thm:unconditional_blockings_product_decompositions}
\begin{equation}\label{eq:connection_blockings;abstract_concrete}
E_{k,\eta} = \bigcup_{i \in J_{k}} I_{i_{1},\eta_{1}} \times \ldots \times I_{i_{n},\eta_{n}}.
\end{equation}

\begin{theorem}[Littlewood--Paley for $\mathcal{A}_p$-weights]\label{thm:littlewood-paley}
		Let $p \in (1, \infty)$ and $\omega \in \mathcal{A}_p(\IR^n; \mathcal{R}_n)$.
For a UMD space $X$ one has the Littlewood--Paley decompositions:
\begin{thm_enum}%
\item\label{item:thm:littlewood-paley;blocking} $\Delta = (\Delta[E])_{E \in \mathcal{E}_{n}}$ forms an unconditional decomposition of $L^{p}_{\omega}(\IR^{n};X)$ with $\mathrm{U}^{\pm}$-constants $C^{\pm}_{\Delta} \lesssim_{X,p,n}[\omega]_{\mathcal{A}_p(\IR^n;\mathcal{R}_n)}^{(n+1)\max\{1,\frac{1}{p-1}\}}$.
\item\label{item:thm:littlewood-paley;product} If $X$ additionally has Pisier's property $(\alpha)$, then $\Delta = (\Delta[I])_{I \in \mathcal{I}_{n}}$ forms an unconditional decomposition of $L^{p}_{\omega}(\IR^{n};X)$ with $\mathrm{U}^{\pm}$-constants $C^{\pm}_{\Delta} \lesssim_{X,p,n}[\omega]_{\mathcal{A}_p(\IR^n;\mathcal{R}_n)}^{2n\max\{1,\frac{1}{p-1}\}}$.
\end{thm_enum}
\end{theorem}
\begin{proof}
Part \ref{item:thm:littlewood-paley;blocking} follows from a combination of Theorem~\ref{thm:unconditional_blockings_product_decompositions}, Lemma~\ref{lem:Littlewood-Paley_coordinate} and Lemma~\ref{lem:cutoffs}, where we use that $X^{*}$ is a UMD space and that the dual family $\Delta^{*}$ is of the same form on $[L^{p}_{\omega}(\IR^{n};X)]^{*} = L^{p'}_{\omega'}(\IR^{n};X^{*})$ with $\omega'=\omega^{-1/(p-1)} \in \mathcal{A}_{p'}(\IR^n; \mathcal{R}_n)$.

Part \ref{item:thm:littlewood-paley;product} directly follows from a combination of Theorem~\ref{thm:unconditional_product_decompositions}, Remark~\ref{rmk:thm:unconditional_product_decompositions;K-convex} and Lemma~\ref{lem:Littlewood-Paley_coordinate}.
\end{proof}

The argumentation used in \ref{item:thm:littlewood-paley;product} is the usual one in case of Pisier's property $(\alpha)$ and basically goes back to \cite{Zim89}, see the proof of Theorem~\ref{thm:unconditional_product_decompositions} and the references given there.
The use of the abstract blocking result Theorem~\ref{thm:unconditional_blockings_product_decompositions} in \ref{item:thm:littlewood-paley;blocking} is due to \cite[Section~3.5]{Wit00} in the periodic setting and can also be found in~\cite{Hyt07b} on anisotropic multipliers.
An alternative approach would be the original one by \cite{Zim89}, using a multi-dimensional Mikhlin theorem (in the spirit of Lemma~\ref{lem:Littlewood-Paley_coordinate}). For this one could use \cite[Proposition~3.2]{MeyVer15}, a Mikhlin theorem for $L^{p}_{\omega}(\IR^{n};X)$ with $\omega \in \mathcal{A}_p(\IR^n; \mathcal{Q}_n)$, obtained from extrapolation via Theorem~\ref{thm:extrapolation} from Theorem~\ref{thm:mikhlin}. However, in the anisotropic case (that we will also consider) a suitable version of Theorem~\ref{thm:extrapolation} is not available.

\subsection{Mikhlin multiplier theorems}\label{subsec:mf_Mikhlin}

The following theorem, which extends~\cite[Theorems 4.4 \& 4.5]{StrWei07} to the two-weighted setting, is a consequence of the generic Theorem~\ref{thm:generic_fm_thm} and the Littlewood--Paley decompositions from Theorem~\ref{thm:littlewood-paley}. Recall the collection of rectangles $\mathcal{I}_{n}$ and $\mathcal{E}_{n}$ introduced before Theorem~\ref{thm:littlewood-paley}.

	\begin{theorem}\label{thm:mikhlin_weighted}
		Let $X$ and $Y$ be UMD spaces, $p \in (1,\infty)$, and $\omega,\sigma \in \mathcal{A}_p(\IR^n; \mathcal{R}_n)$ with $[\omega, \sigma]_{\mathcal{A}_p(\IR^n; \mathcal{R}_n)} < \infty$.
		\begin{thm_enum}%
			\item\label{item:thm:mikhlin_weighted;UMD} Let $m \in L^{\infty}(\IR^{n};\mathcal{B}(X,Y))$ be such that $\partial^{\alpha}m_{|E}$ is continuous for each $E \in \mathcal{E}_{n}$ and $|\alpha|_{\infty} \leq 1$. If
    			\begin{equation*}
    				\norm{m}_{\mathcal{R}\mathfrak{M}(\mathcal{E}_{n})} := \sup_{|\alpha|_{\infty} \le 1} \mathcal{R}\{\abs{\xi}^{\abs{\alpha}} \partial^{\alpha} m_{|E^{\circ}}(\xi): E \in \mathcal{E}_{n}, \xi \in E\} < \infty,
    			\end{equation*}
    		    then $m \in \mathcal{M}_p^n((X,\sigma),(Y,\omega))$ with
    			\begin{equation*}
					\norm{T_m}_{\mathcal{B}(L^{p}_{\sigma}(\IR^{n};X),L^{p}_{\omega}(\IR^{n};Y))} \lesssim_{X,Y,p,n,\sigma,\omega} \norm{m}_{\mathcal{R}\mathfrak{M}(\mathcal{E}_{n})}.
    			\end{equation*}
			\item\label{item:thm:mikhlin_weighted;UMDalpha} Suppose further that $X$ and $Y$ have Pisier's property~$(\alpha)$. If $\mathscr{M} \subset L^{\infty}(\IR^{n};\mathcal{B}(X,Y))$ is such that $\partial^{\alpha}m_{|I^{\circ}}$ is continuous for each $m \in \mathscr{M}$, $I \in \mathcal{I}_{n}$ and $|\alpha|_{\infty} \leq 1$ and
    			\begin{equation*}
    				\norm{\mathscr{M}}_{\mathcal{R}\mathfrak{M}(\mathcal{I}_{n})} := \sup_{|\alpha|_{\infty} \le 1} \mathcal{R}\{\xi^{\alpha} \partial^{\alpha} m_{|I^{\circ}}(\xi): m \in \mathscr{M}, I \in \mathcal{I}_{n}, \xi \in E\} < \infty,
                \end{equation*}
                	then $\mathscr{M} \subset \mathcal{M}^n_p((X,\sigma),(Y,\omega))$ and one has in $\mathcal{B}(L^{p}_{\sigma}(\IR^{n};X),L^{p}_{\omega}(\IR^{n};Y))$
    			\begin{equation*}
					\mathcal{R}\{ T_{m} : m \in \mathscr{M} \} \lesssim_{X,Y,p,n,\sigma,\omega} \norm{\mathscr{M}}_{\mathcal{R}\mathfrak{M}(\mathcal{I}_{n})}.
    			\end{equation*}
		\end{thm_enum}
	\end{theorem}

\begin{remark}
Following the steps of the proof, the power dependency on the weight characteristics can be determined explicitly: indeed, in \ref{item:thm:mikhlin_weighted;UMD} we have
\begin{align*}
C_{X,Y,p,n,\sigma,\omega} & \lesssim_{X,Y,p,n}
 [\omega]_{\mathcal{A}_p(\IR^n;\mathcal{R}_n)}^{(n+1)\max\{1,\frac{1}{p-1}\}} [\sigma]_{\mathcal{A}_p(\IR^n;\mathcal{R}_n)}^{(n+1)\max\{1,\frac{1}{p-1}\}} \\
 & \left([\omega, \sigma]_{\mathcal{A}_p(\IR^n; \mathcal{R}_n)}^{1/p} ([\omega]_{\mathcal{A}_p(\IR^n;\mathcal{R}_n)}^{1-\frac{1}{p}} + [\sigma]_{\mathcal{A}_p(\IR^n;\mathcal{R}_n)}^{\frac{1}{p}}) \right)^{n}
\end{align*}
and \ref{item:thm:mikhlin_weighted;UMDalpha} in we have
\[
C_{X,Y,p,n,\sigma,\omega} \lesssim_{X,Y,p,n}
\left( [\omega]_{\mathcal{A}_p(\IR^n;\mathcal{R}_n)}^{2\max\{1,\frac{1}{p-1}\}} [\sigma]_{\mathcal{A}_p(\IR^n;\mathcal{R}_n)}^{2\max\{1,\frac{1}{p-1}\}} [\omega, \sigma]_{\mathcal{A}_p(\IR^n; \mathcal{R}_n)}^{1/p} ([\omega]_{\mathcal{A}_p(\IR^n;\mathcal{R}_n)}^{1-\frac{1}{p}} + [\sigma]_{\mathcal{A}_p(\IR^n;\mathcal{R}_n)}^{\frac{1}{p}}) \right)^{n}.
\]
However, it is known that the obtained powers are far from optimal for the class of Calderón--Zygmund operators~\cite{Hyt12b}, e.g.\ the Hilbert transform. This loss of exactness stems from our approach based on the Littlewood--Paley decompositions.
\end{remark}

\begin{proof}
By Theorems \ref{thm:generic_fm_thm} and \ref{thm:littlewood-paley}, we only need to check that $\{ m\mathds{1}_{E}: E \in \mathcal{E}_{n} \}$ and $\{ m\mathds{1}_{I}: m \in \mathscr{M}, I \in \mathcal{I}_{n} \}$ are of uniformly $\mathcal{R}$-bounded variation in \ref{item:thm:mikhlin_weighted;UMD} and \ref{item:thm:mikhlin_weighted;UMDalpha}, respectively.

In connection with the representation \eqref{eq:def:uniform_R-bdd_var} in the definition of uniformly $\mathcal{R}$-bounded variation, let us note the following.
Let $I=[a_1,b_1] \times \cdots \times [a_n,b_n] \in \mathcal{R}_{n}$ and let $f\colon I \to Z$ be a continuous function to some Banach space $Z$ whose partial derivatives $\partial^{\alpha}f$, $|\alpha|_{\infty} \leq 1$, exist and are continuous on $I$.
For each $\alpha \in \{0,1\}^{n}$ and $j \in \{1,\ldots,n\}$, let $I_{\alpha,j} \coloneqq \{a_{j},b_{j}\}$ if $\alpha_{j}=0$ and $I_{\alpha,j} := (a_{j},b_{j})$ if $\alpha_{j}=1$, and let $\nu_{I_{\alpha,j}} \coloneqq \delta_{a_{j}}-\delta_{b_{j}}$ if $\alpha_{j}=0$ and $\nu_{I_{\alpha,j}} \coloneqq \mathds{1}_{(a_{j},b_{j})}\lambda^{1}$ if $\alpha_{j}=1$.
For each $\alpha \in \{0,1\}^{n}$, let $I_{\alpha} := \prod_{j=1}^{l}I_{\alpha,j}$ and $\nu_{I_{\alpha}} \coloneqq \otimes_{j=1}^{n}\nu_{I_{\alpha,j}}$.
Extending $f$ by zero to $\IR^{n}$, one has by the fundamental theorem of calculus
\begin{equation}\label{eq:thm:mikhlin_weighted;note_rep}
f(\xi) = \sum_{|\alpha|_{\infty} \leq 1} \int_{(-\infty,\xi]} (\partial^{\alpha}f)\mathds{1}_{I_{\alpha}} \ud \nu_{I,\alpha}, \qquad \xi \in \IR^{n} \setminus (I \setminus I_{\circ}),
\end{equation}
where $I_{\circ} = [a_1,b_1) \times \cdots \times [a_n,b_n)$.
For \ref{item:thm:mikhlin_weighted;UMD} we can use \eqref{eq:thm:mikhlin_weighted;note_rep} to obtain
\begin{align*}
(m\mathds{1}_{E})(\xi)
& = \sum_{|\alpha|_{\infty} \leq 1} \int_{(-\infty,\xi]}|\eta|^{|\alpha|}\partial^{\alpha}m_{|E}(\eta)\mathds{1}_{E_{\alpha}}(\eta) |\eta|^{-|\alpha|}\ud\nu_{E,\alpha}(\eta) \\
& =  \int_{(-\infty,\xi]} \sum_{|\beta|_{\infty}\leq1}|\eta|^{|\beta|}\partial^{\beta}m_{|E}(\eta)\mathds{1}_{E_{\beta}}(\eta) \:\:\sum_{|\alpha|_{\infty} \leq 1}|\eta|^{-|\alpha|}\ud\nu_{E,\alpha}(\eta)
\end{align*}
for a.e.\ $\xi \in \IR^{n}$,
where the second equality follows from disjointness of supports.
For the symbol $m\mathds{1}_{E}$ we can thus take $\tau_{m,E}(\eta) := \sum_{|\beta|_{\infty}\leq1}|\eta|^{|\beta|}\partial^{\beta}m_{|E}(\eta)\mathds{1}_{E_{\beta}}(\eta)$ and $\mathrm{d}\mu_{m,E}(\eta ):= \sum_{|\alpha|_{\infty} \leq 1}|\eta|^{-|\alpha|}\ud\nu_{E,\alpha}(\eta)$ in the representation \eqref{eq:def:uniform_R-bdd_var}.
Since
\[
\mathcal{R}(\{\tau_{m,E}(\eta):\eta \in \IR^{n}, E \in \mathcal{E}_{n}\}) \leq \norm{m}_{\mathcal{R}\mathfrak{M}(\mathcal{E}_{n})}
\]
and for $E = E_{ln + r,\eta}$ one has
\begin{align*}
\normalnorm{\mu_{m,E}}
& \leq \sum_{|\alpha|_{\infty} \leq 1}\norm{|\eta|^{-|\alpha|}\ud\nu_{E,\alpha}(\eta)} \eqsim \sum_{|\alpha|_{\infty} \leq 1}2^{-l|\alpha|}\prod_{j:\alpha_{j=1}}\normalnorm{\mathds{1}_{(a_{j,E},b_{j,E})}\lambda^{1}}  \\
& \leq \sum_{|\alpha|_{\infty} \leq 1}2^{-l|\alpha|}(2^{l+1})^{|\alpha|} \lesssim_{n} 1,
\end{align*}
it follows that
\[
\mathrm{var}_{\mathcal{R}}(\{ m\mathds{1}_{E}: E \in \mathcal{E}_{n} \}) \lesssim_{n} \norm{m}_{\mathcal{R}\mathfrak{M}(\mathcal{E}_{n})}.
\]

In case of \ref{item:thm:mikhlin_weighted;UMDalpha} one similarly gets that $m\mathds{1}_{I}$ satisfies \eqref{eq:def:uniform_R-bdd_var} with
\begin{equation*}
	\tau_{m,I}(\eta) \coloneqq \sum_{|\beta|_{\infty}\leq1}\eta^{\beta}\partial^{\beta}m_{|I}(\eta)\mathds{1}_{I_{\beta}}(\eta) \qquad \text{and} \qquad \mathrm{d}\mu_{m,I}(\eta ) \coloneqq \sum_{|\alpha|_{\infty} \leq 1}\eta^{-\alpha}\ud\nu_{I,\alpha}(\eta).
\end{equation*}
Then,
\[
\mathcal{R}(\{\tau_{m,I}(\eta):m \in \mathscr{M}, \eta \in \IR^{n}, I \in \mathcal{I}_{n}\}) \leq \norm{\mathscr{M}}_{\mathcal{R}\mathfrak{M}(\mathcal{I}_{n})}
\]
and
\begin{align*}
\normalnorm{\mu_{m,I}}
& \leq \sum_{|\alpha|_{\infty} \leq 1} \normalnorm{\eta^{-\alpha}\ud\nu_{I,\alpha}(\eta)}
\lesssim \sum_{|\alpha|_{\infty} \leq 1}\prod_{j:\alpha_{j=1}}\int_{(a_{j,I},b_{j,I})}\eta_{j}^{-1}\mathrm{d}\eta_{j}
= \sum_{|\alpha|_{\infty} \leq 1}\log^{|\alpha|}(2) \lesssim_{n} 1,
\end{align*}
so that
\[
\mathrm{var}_{\mathcal{R}}(\{ m\mathds{1}_{I}: m \in \mathscr{M}, I \in \mathcal{I}_{n} \}) \lesssim_{n} \norm{\mathscr{M}}_{\mathcal{R}\mathfrak{M}(\mathcal{I}_{n})}. \qedhere
\]
\end{proof}

We finally state an anisotropic version of Theorem~\ref{thm:mikhlin_weighted}.\ref{item:thm:mikhlin_weighted;UMD} in the weighted mixed-norm setting, extending \cite[Theorem~3.2]{Hyt07b} (see also~\cite[Section~7]{Hyt07c}) to the weighted setting. Such a result is an important tool in the weighted maximal $L^{q}$-$L^{p}$-regularity approach to parabolic problems (see~\cite{Lin17a}).
Let us first introduce the anisotropic setting.
Given $\vec{a} \in (0,\infty)^{n}$,
we define the $\vec{a}$-anisotropic distance function $|\,\cdot\,|_{\vec{a}}$ on $\IR^{n}$ by the formula
\begin{equation*}
|x|_{\vec{a}} := \left(\sum_{j=1}^{n}|x_{j}|^{2/a_{j}}\right)^{1/2}.
\end{equation*}
We furthermore define an $\vec{a}$-anisotropic version $\mathcal{E}^{\vec{a}}_{n} = \{ E^{\vec{a}}_{k,\eta} : (k,\eta) \in \IZ \times \{-1,1\}^{n} \}$ of the decomposition $\mathcal{E}_{n}$ (introduced before Theorem~\ref{thm:littlewood-paley}) for $l \in \IZ$ and $r \in \{0, \ldots, n-1 \}$ by
\[
E^{\vec{a}}_{ln+r,\eta} \coloneqq \prod_{j=1}^{r}\eta_{j}[0,2^{a_{j}(l+1)}] \times \eta_{r}[2^{a_{j}l},2^{a_{j}(l+1)}] \times \prod_{j=r+2}^{n}\eta_{i}[0,2^{a_{j}l}].
\]

Let us next introduce the weighted mixed-norm setting.
Suppose that $n =  n_{1} + \ldots + n_{l}$ with $n_{1},\ldots,n_{l} \in \IZ_{\geq 1}$, $l \in \IN$, and view $\IR^{n}$ as $\IR^{n} = \IR^{n_{1}} \times \ldots \times \IR^{n_{l}}$.
For $x \in \IR^{n}$ we accordingly write $x = (x_{1},\ldots,x_{l})$ with $x_{j}=(x_{j,1},\ldots,x_{j,n_{j}})$, where $x_{j} \in \IR^{n_{j}}$ and $x_{j,i} \in \IR$ $(j=1,\ldots,l; i=1,\ldots,n_j)$.
Given $\vec{p} \in (1,\infty)^{l}$ and $\vec{\omega} \in \prod_{j=1}^{l}\mathcal{A}_{p_{j}}(\IR^{n_{j}}; \mathcal{R}_{n_{j}})$,
we define associated the weighted mixed-norm Bochner space $L^{\vec{p}}_{\vec{\omega}}(\IR^{n};X)$ as the Banach space of all Bochner measurable $f\colon \IR^n \to X$ satisfying
\[
\norm{f}_{L^{\vec{p}}_{\vec{\omega}}(\IR^{n};X)} :=
 \left( \int_{\IR^{n_{l}}} \ldots \left(\int_{\IR^{n_{1}}}\norm{f(x)}_{X}^{p_{1}}\omega_{1}(x_{1})dx_{1} \right)^{p_{2}/p_{1}} \ldots \omega_{l}(x_{l})dx_{l} \right)^{1/p_{l}} < \infty.
\]
We denote by $\mathcal{M}^n_{\vec{p}}((X,\vec{\sigma}),(Y,\vec{\omega}))$ the set of all Fourier multipliers $L^{\vec{p}}_{\vec{\sigma}}(\IR^{n};X) \to L^{\vec{p}}_{\vec{\omega}}(\IR^{n};Y)$.

\begin{theorem}\label{thm:mikhlin_weighted_mixed-norm_anisotropic}
		Let $X$ and $Y$ be UMD spaces, $\vec{p} \in (1,\infty)^{l}$, and $\vec{\omega},\vec{\sigma} \in \prod_{j=1}^{l}\mathcal{A}_{p_{j}}(\IR^{n_{j}}; \mathcal{R}_{n_{j}})$ with $[\omega_{j}, \sigma_{j}]_{\mathcal{A}_{p_{j}}(\IR^{n_{j}}; \mathcal{R}_{n_{j}})} < \infty$ for $j=1,\ldots,l$.
Let $m \in L^{\infty}(\IR^{n};\mathcal{B}(X,Y))$ be such that $\partial^{\alpha}m_{|E^{\circ}}$ is continuous for all $E \in \mathcal{E}^{\vec{a}}_{n}$ and $|\alpha|_{\infty} \leq 1$. If
    			\begin{equation*}
    				\norm{m}_{\mathcal{R}\mathfrak{M}(\mathcal{E}^{\vec{a}}_{n})} := \sup_{|\alpha|_{\infty} \le 1} \mathcal{R}\{|\xi|_{\vec{a}}^{\abs{\alpha}} \partial^{\alpha} m_{|E^{\circ}}(\xi): E \in \mathcal{E}_{n}, \xi \in E\} < \infty,
    			\end{equation*}
    		    then $m \in \mathcal{M}^n_{\vec{p}}((X,\vec{\sigma}),(Y,\vec{\omega}))$ with
    			\begin{equation*}
					\norm{T_m}_{\mathcal{B}(L^{\vec{p}}_{\vec{\sigma}}(\IR^{n};X),L^{\vec{p}}_{\vec{\omega}}(\IR^{n};Y))} \lesssim_{X,Y,\vec{p},n,\vec{\sigma},\vec{\omega}} \norm{m}_{\mathcal{R}\mathfrak{M}(\mathcal{E}^{\vec{a}}_{n})}.
\end{equation*}
\end{theorem}
\begin{proof}
In the same way as in Theorem~\ref{thm:littlewood-paley} it can be shown that $\Delta =\{ \Delta[E] : E \in \mathcal{E}^{\vec{a}}_{n}\}$ defines an unconditional decomposition of $L^{\vec{p}}_{\vec{\omega}}(\IR^{n};X)$ with $\mathrm{U}^{\pm}$-constants
\[
C^{\pm}_{\Delta} \lesssim_{X,\vec{p},n}\sum_{i=1}^{l}\prod_{j=1}^{l}
[\omega_{j}]_{\mathcal{A}_{p_{j}}(\IR^{n_{j}};\mathcal{R}_{n_{j}})}^{(n_{j}+\delta_{i,j})\max\{1,\frac{1}{p_{j}-1}\}},
\]
where $\delta_{i,j}$ denotes the Kronecker delta.
In the same way as in Lemma~\ref{lem:cutoffs} it can be shown that $\{ \Delta[R]: R \in \mathcal{R}_{n} \}$ is a bounded family in  $\mathcal{B}(L^{\vec{p}}_{\vec{\sigma}}(\IR^n; X), L^{\vec{p}}_{\vec{\omega}}(\IR^n;X))$ with $\mathcal{R}$-bound
\[
\mathcal{R}\{\Delta[R]: R \in \mathcal{R}_{n} \} \lesssim_{X,\vec{p}} \prod_{j=1}^{l}\left( [\omega_{j}, \sigma_{j}]_{\mathcal{A}_{p_{j}}(\IR^{n_{j}}; \mathcal{R}_{n_{j}})}^{1/p_{j}} ([\omega_{j}]_{\mathcal{A}_p(\IR^{n_{j}};\mathcal{R}_{n_{j}})}^{1-\frac{1}{p_{j}}} + [\sigma_{j}]_{\mathcal{A}_p(\IR^{n_{j}};\mathcal{R}_{n_{j}})}^{\frac{1}{p_{j}}}) \right)^{n_{j}}.
\]
We may thus apply the generic Theorem~\ref{thm:generic_fm_thm} and the proof of the theorem is now completed in the same way as in Theorem~\ref{thm:mikhlin_weighted}.\ref{item:thm:mikhlin_weighted;UMD}.
\end{proof}
	
\section{Fourier Multipliers for Cubular $\mathcal{A}_{p}$-weights}\label{sec:cubular}

	The approach for weighted estimates of multipliers based on Littlewood--Paley theory gives by its very own nature only results for the more restrictive and one-dimensional behaving class $\mathcal{A}_p(\IR^n; \mathcal{R}_n)$. Naturally, it is also very desirable to obtain results for the weaker class $\mathcal{A}_p(\IR^n; \mathcal{Q}_n)$. This is indeed possible if one works with Hörmander type conditions instead of the Littlewood--Paley decomposition. Hence, we pass from the multiplier perspective to the perspective of singular integrals. Nevertheless, as a consequence we will obtain weighted results for Fourier multipliers.

In this section we will use the following Banach space geometric property:
	\begin{definition}
		A complex Banach space $X$ has \emph{Fourier type $t \in [1,2]$} if the vector-valued Fourier transform $\mathcal{F}\colon \mathcal{S}(\IR^n;X) \to \mathcal{S}(\IR^n;X)$ extends for one (or equivalently all) $n \in \IN$ to a bounded operator $L^{t}(\IR^n;X) \to L^{t'}(\IR^n;X)$.
	\end{definition}
	
	Note that every Banach space has Fourier type $1$ by the Riemann--Lebesgue lemma and that Fourier type $t$ implies Fourier type $s$ for all $s \in [1,t]$. Further, Kwapień showed that a Banach space has Fourier type $2$ if and only if it is isomorphic to a Hilbert space~\cite[Proposition~4.1]{Kwa72}.

We will furthermore use the weight characteristic $[\omega,\sigma]_{\mathcal{A}^{r}_{p}(\IR^{n};\mathcal{Q}_{n})}$ defined by
\[
[\omega,\sigma]_{\mathcal{A}^{r}_{p}(\IR^{n};\mathcal{Q}_{n})} := \sup_{Q \in \mathcal{Q}_{n}}\left( \frac{\sigma^{r}(Q)}{|Q|}\right)^{\frac{1}{r}-\frac{1}{p}}\left(\frac{\omega(Q)}{|Q|} \right)^{1/p},
\]
where $r,p \in (1,\infty)$.

\subsection{Domination by sparse operators}

	We first show that certain multiplier operators are dominated by rather easy operators, namely sparse operators.
	
	\begin{definition}
		A collection $\mathcal{S}$ of cubes in $\IR^n$ is called \emph{sparse} if for some $\eta > 0$ there exists a pairwise disjoint collection $(E_Q)_{Q \in \mathcal{S}}$ such that for every $Q \in \mathcal{S}$ the set $E_Q$ is a measurable subset of $Q$ and satisfies $\normalabs{E_Q} \ge \eta \abs{Q}$. Given a sparse family $\mathcal{S}$ and $r \in [1, \infty)$, we define for non-negative measurable functions $f$ the associated \emph{sparse operator} as
		\begin{equation*}
			\mathcal{A}_{r,\mathcal{S}}f = \sum_{Q \in \mathcal{S}} \biggl( \frac{1}{\abs{Q}} \int_Q f^r \biggr)^{1/r} \mathds{1}_Q.
		\end{equation*}
	\end{definition}
The \emph{standard dyadic grid} on $\IR^n$ is the collection $\mathcal{D}$ of cubes $\{ 2^{-j} ([0,1)^n + m): j \in \IZ, m \in \IZ^n \}$. Further, for given $(\omega_k)_{k \in \IZ} \in (\{ 0, 1 \}^n)^{\IZ}$ we define a \emph{shifted dyadic grid} $\mathcal{D}^{\omega}$ as the collection
	\begin{equation*}
		\mathcal{D}^{\omega} \coloneqq \Bigl\{ Q + \sum_{j: 2^{-j} < \ell(Q)} \omega_j 2^{-j}: Q \in \mathcal{D} \Bigr\}.
	\end{equation*}

Having domination by sparse operators, the following theorem subsequently yields weighted estimates.
	\begin{theorem}[\cite{HytLi18}]\label{thm:sharp_estimate_sparse}
		Let $r \in (1, \infty)$ and $\mathcal{S}$ a sparse family of cubes out of a fixed shifted dyadic system. Then for $p \in (r,\infty)$ and all measurable $f\colon \IR^n \to \IR_{\ge 0}$
		\begin{equation*}
			\normalnorm{\mathcal{A}_{r,\mathcal{S}}(f \sigma)}_{L^p_{\omega}(\IR^n)} \lesssim_{p,r,\mathcal{S}} [\omega, \sigma]_{\mathcal{A}_p^r(\IR^n;\mathcal{Q}_n)} ([\omega]_{\mathcal{A}_{\infty}(\IR^n;\mathcal{Q}_n)}^{1/p'} + [\sigma^r]_{\mathcal{A}_{\infty}(\IR^n;\mathcal{Q}_n)}^{1/p}) \norm{f}_{L^p_{\sigma^r}(\IR^n)}.
		\end{equation*}
	\end{theorem}
The above theorem is actually a reformulation of \cite[Theorem~1.1]{HytLi18} (also see \cite[Theorem~1.1]{HytFack18}).
	
	For an operator $T$ mapping $L^p(\IR^n;X)$ into vector-valued measurable functions and some $k \in \IN$ we define its \emph{grand maximal truncated operator} as
	\begin{equation*}
		\mathcal{M}_{T,k}f(x) = \sup_{Q \ni x} \esssup_{y \in Q} \norm{T(f \mathds{1}_{((2k+1)Q)^c})(y)}.
	\end{equation*}
	Here $(2k+1)Q$ denotes the cube with the same center and side length $2k+1$ times that of $Q$. Recall that a Bochner measurable function $f\colon \IR^n \to \IC$ lies in $L^{p,\infty}(\IR^n;X)$ for $p \in [1, \infty)$ if there exists $C \ge 0$ with
	\begin{equation*}
		\abs{\{ x \in \IR^n: \norm{f(x)} > \lambda \}} \le \frac{C^p}{\lambda^p} \qquad \text{for all } \lambda > 0.
	\end{equation*}
	Further, $\norm{f}_{L^{p,\infty}(\IR^n;X)}$ is the smallest $C$ for which the above estimate holds. One has the following general domination theorem by Lerner~\cite[Theorem~4.2]{Ler16} which only deals with the scalar case and the choice $k = 1$. However, a very similar argument does work in this more general setting as well.
	
	\begin{theorem}\label{thm:lerner}
		Let $X,Y$ be Banach spaces and $T\colon L^q(\IR^n;X) \to L^{q,\infty}(\IR^n;Y)$ be linear and bounded for some $q \in [1, \infty)$. Further suppose that $\mathcal{M}_T\colon L^r(\IR^n;X) \to L^{r,\infty}(\IR^n)$ for some $r \in [q,\infty)$. Then for every compactly supported $f \in L^r(\IR^n;X)$ and all $k \in \IN$ there exist sparse families $\mathcal{S}_1, \ldots, \mathcal{S}_{3^n}$ of cubes out of different shifted dyadic grids such that almost everywhere
			\[
				\norm{Tf(x)} \lesssim (\norm{T}_{L^q \to L^{q,\infty}} + \normalnorm{\mathcal{M}_{T,k}}_{L^r \to L^{r,\infty}}) \sum_{j=1}^{3^n} \mathcal{A}_{r,\mathcal{S}_j} \norm{f}(x).
			\]
		Here the implicit constant only depends on $k$, $n$, $q$ and $r$.
	\end{theorem}
	
	In the following we verify the assumptions of Lerner's domination theorem for certain Fourier multipliers. The geometric property Fourier type of a Banach space (see Section~\ref{sec:r-boundedness}) plays a role in our estimates.
	
In the next result, a vector-valued adaption of~\cite[Lemma~1]{KurWhe79}, we establish a connection between Hörmander type conditions on the multiplier and estimates on the kernel. Let us fix the setting.
	
		Let $m \in L^{\infty}(\IR^{n};\mathcal{B}(X,Y))$. We fix some $\phi \in C_c^{\infty}(\IR^n)$ with support contained in $\frac{1}{2} \le \abs{\xi} \le 2$ and inducing a partition of unity, i.e.\ $\sum_{j=-\infty}^{\infty} \phi(2^{-j} \xi) =  1$ for all $\xi \neq 0$ (see also Lemma~\ref{lem:Littlewood-Paley_coordinate}). For $m_j(\xi) = m(\xi) \phi(2^{-j}\xi)$ we have $m(\xi) = \sum_{j=-\infty}^{\infty} m_j(\xi)$ for $\xi \neq 0$. The multipliers $m_j$ are integrable and compactly supported. Therefore the corresponding kernel $k_j = \mathcal{F}^{-1} m_j$ is a bounded smooth function. For $N \in \IN$ we consider the approximative kernels $K_N = \sum_{j=-N}^N k_j$. For all $f \in \mathcal{S}(\IR^n;X)$ one then has
		\begin{equation*}
			K_N * f = \mathcal{F}^{-1} \biggl(\sum_{j=-N}^N m_j \cdot \mathcal{F}f \biggr).
		\end{equation*}

In the lemma below there is the H\"ormander type condition \eqref{eq:lem:kernel_estimate} on the the multipier $m \in L^{\infty}(\IR^{n};\mathcal{B}(X,Y))$. For convenience we use the convention that when we write down a condition like \eqref{eq:lem:kernel_estimate} it is implicitely assumed that the expression in \eqref{eq:lem:kernel_estimate} is well-defined in a natural way; we require the distribution $\partial^{\alpha}m$ to be regular on $\IR^{n} \setminus \{0\}$ in the sense that $\partial^{\alpha}m_{|\IR^{n} \setminus \{0\}} \in L^{1}_{\mathrm{loc}}(\IR^{n} \setminus \{0\};\mathcal{B}(X,Y))$ (or actually that $\partial^{\alpha}m \, u$ is a regular distribution on $\IR^{n} \setminus \{0\}$ for all $u \in X_{0}$ in the specific case of the lemma below).

		\begin{lemma}\label{lem:kernel_estimate}
		Let $X, Y$ be Banach spaces and $m \in L^{\infty}(\IR^{n};\mathcal{B}(X,Y))$. Assume that $Y$ has non-trivial Fourier type $t > 1$ and that there exist $s \in [t,2]$ and $l \in (\frac{n}{t},n] \cap \IN$ such that for some $C \ge 0$ one has for all $R > 0$ and all $u$ in a subset $X_0 \subset X$
		\begin{equation}\label{eq:lem:kernel_estimate}
			\biggl( R^{s\abs{\alpha} - n} \int_{R < \abs{\xi} < 2R} \norm{\partial^{\alpha} m(\xi)u}^s \d \xi \biggr)^{1/s} \le C \norm{u} \qquad \text{for all } \abs{\alpha} \le l.
		\end{equation}
		If there exists $d \in (\frac{n}{t},\frac{n}{t}+1) \cap \IN$, then for $p \in [1,t']$, $R >0$, $\abs{y} \le \frac{R}{2}$ and $u \in X_0$ we have uniformly in $N \in \IN$ the kernel estimate
		\begin{equation*}
			\biggl( \int_{R < \abs{x} < 2R} \norm{[K_N(x-y) - K_N(x)]u}^{p} \d x \biggr)^{1/p} \lesssim R^{\frac{n}{p}-\frac{n}{t'}-d} \abs{y}^{d-\frac{n}{t}} \norm{u},
		\end{equation*}
 where the implicit constant only depends on $n$, $t$, $s$, $l$, $C$ and $\norm{\mathcal{F}}_{L^{t}(\IR^{n};Y) \to L^{t'}(\IR^{n};Y)}$.
		In particular, for $p \in [1, t']$ the kernels satisfy uniformly the pointwise $p$-Hörmander condition, i.e.\ there exists $(a_k)_{k \in \IN} \in \ell^1$ such that for all $k \in \IN$, $u \in X_0$ and $y \in \IR^n$
			\begin{equation*}
				\biggl( \int_{2^k \abs{y} < \abs{x} < 2^{k+1} \abs{y}} \norm{[K_N(x-y) - K_N(x)]u}^{p} \d x \biggr)^{1/p} \le a_k (2^k \abs{y})^{-\frac{n}{p'}} \norm{u}.
		\end{equation*}
	\end{lemma}
	\begin{proof}
		It follows from the assumptions that for $\abs{\alpha} \le l$, $j \in \IZ$ and $q \le s$
		\begin{equation}
			\label{eq:localized_kernel_estimate}
			\begin{split}
    			\MoveEqLeft \biggl( \int_{\IR^n} \normalnorm{\partial^{\alpha} m_j(\xi) u}^q \d\xi \biggr)^{1/q} \lesssim 2^{j(\frac{n}{q}-\frac{n}{s})} \biggl( \int_{\IR^n} \normalnorm{\partial^{\alpha} m_j(\xi) u}^s \d\xi \biggr)^{1/s} \\
				& \lesssim 2^{j(\frac{n}{q}-\frac{n}{s})} \sum_{\beta \le \alpha} 2^{-j \abs{\alpha - \beta}} \biggl( \int_{2^{j-1} \le \abs{\xi} \le 2^{j+1}} \normalnorm{\partial^{\beta} m(\xi) u}^s \d\xi \biggr)^{1/s} \\
    			& \lesssim 2^{j(\frac{n}{q}-\frac{n}{s})} \sum_{\beta \le \alpha} 2^{-j\abs{\alpha - \beta}} 2^{(j-1)(\frac{n}{s}-\abs{\beta})} \norm{u} \lesssim 2^{j(\frac{n}{q}-\abs{\alpha})} \norm{u}.
			\end{split}
		\end{equation}
		By Minkowski's inequality one has
		\begin{align*}
			\MoveEqLeft \biggl( \int_{R < \abs{x} < 2R} \norm{[K_N(x-y) - K_N(x)]u}^{p} \d x \biggr)^{1/p} \\
			& \le \sum_{j=-\infty}^{\infty} \biggl( \int_{R < \abs{x} < 2R} \normalnorm{[k_j(x-y) - k_j(x)]u}^{p} \d x \biggr)^{1/p}.
		\end{align*}
		Thus it suffices to prove suitable estimates for the kernels $k_j u$. We now estimate them separately. On the one hand one has for $R > 0$ and $\abs{y} \le \frac{R}{2}$
		\begin{align*}
			\MoveEqLeft \biggl( \int_{R < \abs{x} < 2R} \norm{[k_j(x-y) - k_j(x)]u}^{p} \d x \biggr)^{1/p} \\
			& \le \biggl( \int_{R < \abs{x} < 2R} \normalnorm{k_j(x-y)u}^{p} \d x \biggr)^{1/p} + \biggl( \int_{R < \abs{x} < 2R} \normalnorm{k_j(x)u}^{p} \d x \biggr)^{1/p} \\
			& \le 2 \biggl( \int_{R/2 < \abs{x} < 5R/2} \normalnorm{k_j(x)u}^{p} \d x \biggr)^{1/p}.
		\end{align*}
		Since $d \le l$, we obtain for the last above term
		\begin{align*}
			\MoveEqLeft \biggl( \int_{R/2 < \abs{x} < 5R/2} \normalnorm{k_j(x)u}^{p} \d x \biggr)^{1/p} \lesssim R^{-d} \biggl( \int_{R/2 < \abs{x} < 5R/2} \normalnorm{\abs{x}^d k_j(x)u}^{p} \d x \biggr)^{1/p} \\
			& \lesssim R^{-d} \sum_{\abs{\alpha} = d} \biggl( \int_{R/2 < \abs{x} < 5R/2} \normalnorm{x^{\alpha} k_j(x)u}^{p} \d x \biggr)^{1/p}.
		\end{align*}
		Recall that by assumption $Y$ has Fourier type $t$ and that $p \le t'$. Hence, for each of the above summands we have by~\eqref{eq:localized_kernel_estimate}
		\begin{equation}\label{eq:kernel_estimate_1}
    		\begin{split}
    			\MoveEqLeft R^{-d} \biggl( \int_{R/2 < \abs{x} < 5R/2} \normalnorm{x^{\alpha} k_j(x)u}^{p} \d x \biggr)^{1/p} \lesssim R^{\frac{n}{p} - \frac{n}{t'}-d} \biggl( \int_{R/2 < \abs{x} < 5R/2} \norm{\mathcal{F}^{-1}(\partial^{\alpha} m_j u)(x)}^{t'} \d x \biggr)^{1/t'} \\
    			& \lesssim R^{\frac{n}{p} - \frac{n}{t'}-d} \biggl( \int_{\IR^n} \normalnorm{\partial^{\alpha} m_j(\xi) u}^t \d\xi \biggr)^{1/t} \lesssim R^{\frac{n}{p} - \frac{n}{t'}-d} (2^j)^{\frac{n}{t}-d} \norm{u}.
    		\end{split}
		\end{equation}
		In the same spirit we can estimate the difference as
		\begin{align*}
			\MoveEqLeft \biggl( \int_{R < \abs{x} < 2R} \normalnorm{[k_j(x-y) - k_j(x)]u}^{p} \d x \biggr)^{1/p} \\
			& \lesssim R^{\frac{n}{p} - \frac{n}{t'}} \biggl( \int_{R < \abs{x} < 2R} \normalnorm{[k_j(x-y) - k_j(x)]u}^{t'} \d x \biggr)^{1/t'} \\
			& \lesssim R^{\frac{n}{p} - \frac{n}{t'} - d} \sum_{\abs{\alpha} = d} \biggl( \int_{R < \abs{x} < 2R} \norm{x^{\alpha} \mathcal{F}^{-1}((e^{iy \cdot} - 1) m_j u)(x)}^{t'} \d x \biggr)^{1/t'} \\
			& \lesssim R^{\frac{n}{p} - \frac{n}{t'} - d} \sum_{\abs{\alpha} = d} \biggl( \int_{\IR^n} \normalnorm{\partial^{\alpha} [(e^{i y \cdot} - 1) m_j](\xi)u]}^t \d\xi \biggr)^{1/t} \\
			& = R^{\frac{n}{p} - \frac{n}{t'} - d} \sum_{\abs{\beta} + \abs{\gamma} = d} \biggl( \int_{\IR^n} \normalnorm{\partial^{\gamma} (e^{i y \cdot} - 1)(\xi) \partial^{\beta} m_j(\xi)u}^{t} \d\xi \biggr)^{1/t}.
		\end{align*}	
		We now estimate the above summands. For $\abs{\gamma} = 0$ and $\abs{\beta} = d$ the inequality $\normalabs{e^{i y \xi}-1} \le \abs{\xi} \abs{y}$ and~\eqref{eq:localized_kernel_estimate} gives
		\begin{align*}
			\MoveEqLeft \biggl( \int_{\IR^n} \normalabs{e^{i \xi y}-1}^t \normalnorm{\partial^{\beta} m_j(\xi) u}^t \d \xi \biggr)^{1/t} \le \biggl( \int_{\IR^n} \abs{\xi}^t \abs{y}^t \normalnorm{\partial^{\beta} m_j(\xi) u}^t \d \xi \biggr)^{1/t} \lesssim \abs{y} 2^{j(\frac{n}{t}-d+1)} \norm{u}.
		\end{align*}
		For $\abs{\gamma} > 0$ we use $\normalabs{\partial^{\gamma} (e^{i y \cdot} - 1)} \le \abs{y}^{\abs{\gamma}}$ together with~\eqref{eq:localized_kernel_estimate} and obtain
		\begin{align*}
			\MoveEqLeft \biggl( \int_{\IR^n} \normalnorm{\partial^{\gamma} (e^{i y \cdot} - 1)(\xi) \partial^{\beta} m_j(\xi)u}^t \d\xi \biggr)^{1/t} \lesssim \abs{y}^{\abs{\gamma}} 2^{j(\frac{n}{t}-\abs{\beta})} \norm{u}.
		\end{align*}
		Adding the two just obtained estimates, we get
		\begin{equation}\label{eq:kernel_estimate_2}
			\begin{split}
				\MoveEqLeft \biggl( \int_{R < \abs{x} < 2R} \normalnorm{[k_j(x-y) - k_j(x)]u}^{p} \d x \biggr)^{1/p} \lesssim R^{\frac{n}{p}-\frac{n}{t'}-d} \sum_{m=1}^{d} \abs{y}^m (2^j)^{\frac{n}{t}-d+m} \norm{u}.
			\end{split}
		\end{equation}
		Finally, putting~\eqref{eq:kernel_estimate_1} and~\eqref{eq:kernel_estimate_2} together and using $\abs{y}^m (2^j)^{n/t-d+m} \le \abs{y} (2^j)^{n/t-d+1}$ for $2^j \le \abs{y}^{-1}$, we have because of $d \in (\frac{n}{t},\frac{n}{t}+1)$ the claimed estimate
		\begin{align*}
			\MoveEqLeft \biggl( \int_{R < \abs{x} < 2R} \norm{[K_N(x-y) - K_N(x)]u}^{p} \d x \biggr)^{1/p} \\
			& \lesssim \sum_{2^j \le \abs{y}^{-1}} R^{\frac{n}{p}-\frac{n}{t'}-d} \abs{y} (2^j)^{\frac{n}{t}+1-d} \norm{u} + \sum_{2^j \ge \abs{y}^{-1}} R^{\frac{n}{p}-\frac{n}{t'}-d} (2^j)^{\frac{n}{t} - d} \norm{u} \lesssim R^{\frac{n}{p}-\frac{n}{t'}-d} \abs{y}^{d-\frac{n}{t}} \norm{u}. \qedhere
		\end{align*}
	\end{proof}	
	
	\begin{remark}\label{rem:classical_hoermander}
		For $p = 1$ the pointwise $p$-Hörmander condition reduces to the pointwise variant of the well-known Hörmander condition, namely
		\begin{equation*}
			\int_{\abs{x} > 2\abs{y}} \norm{[K_N(x-y) - K_N(x)]u} \d x \lesssim \norm{u}.
		\end{equation*}
	\end{remark}
	
	We now use the obtained estimates on the kernel to verify the assumptions of Lerner's domination theorem.
	
	\begin{lemma}\label{lem:domination}
		Let $X,Y$ be Banach spaces and suppose that $X$ has non-trivial Fourier type $t > 1$. Further, let $m \in L^{\infty}(\IR^{n};\mathcal{B}(X,Y))$ be such that $T_m\colon L^q(\IR^n;X) \to L^{q,\infty}(\IR^n;Y)$ for some $q \in (1, \infty)$. Assume that there exist $s \in [t,2]$ and $l \in (\frac{n}{t},n] \cap \IN$ such that for some $C \ge 0$ one has for all $R > 0$ and all $v^* \in \mathcal{N}$ in a subset $\mathcal{N}$ of $Y^*$ norming for $Y$ the estimate
		\begin{equation*}
			\biggl( R^{s\abs{\alpha} - n} \int_{R < \abs{\xi} < 2R} \norm{\partial^{\alpha} m^*(\xi)v^*}^s \d\xi \biggr)^{1/s} \le C \norm{v^*} \qquad \text{for all } \abs{\alpha} \le l.
		\end{equation*}
		If $q \in [t, \infty)$, then for every compactly supported $f \in L^q(\IR^n;X)$ there exist sparse families $\mathcal{S}_1, \ldots, \mathcal{S}_{3^n}$ of cubes out of different shifted dyadic grids such that almost everywhere
		\begin{equation*}
			\norm{T_mf(x)} \lesssim \sum_{j=1}^{3^n} \mathcal{A}_{q,\mathcal{S}_j} \norm{f}(x).
		\end{equation*}
Here the implicit constant only depends on $n$, $s$, $l$, $t$, $\norm{\mathcal{F}}_{L^{t}(\IR^{n};X) \to L^{t'}(\IR^{n};X)}$, $C$, $q$ and $\norm{T}_{L^{q} \to L^{q,\infty}}$.
	\end{lemma}
	
	\begin{proof}		
		As already said, we verify the assumptions of Theorem~\ref{thm:lerner}. The required mapping property for $T = T_m$ is satisfied by our made assumptions. We now show that $\mathcal{M}_{T,k}\colon L^q(\IR^n;X) \to L^{q,\infty}(\IR^n)$ is bounded for some sufficiently large $k \in \IN$. For $f \in L^1(\IR^n;X)$ let $T_N f = K_N * f \in C(\IR^n;Y)$. Fix $x \in \IR^n$ and $Q$ with $x \in Q$. For $y,z \in Q$ we have
		\begin{align*}
			T_N(f \mathds{1}_{((2k+1)Q)^c})(y) & = T_N(f \mathds{1}_{((2k+1)Q)^c})(y) - T_N(f \mathds{1}_{((2k+1)Q)^c})(z) + T_N(f \mathds{1}_{((2k+1)Q)^c})(z).
		\end{align*}
		The values in $y$ and $z$ are comparable. For this note that $y, z \in Q$ implies $\abs{z-y} \le \sqrt{n} \ell(Q)$. Consequently, if $k$ is chosen large, we have $\abs{x-y} \le \frac{1}{4} (2k+1) \ell(Q)$. Then for all $v^* \in \mathcal{N}$
		\begin{equation}\label{eq:difference_estimate_1}
    		\begin{split}
    			\MoveEqLeft \langle v^*, T_N(f \mathds{1}_{((2k+1)Q)^c})(y) - T_N(f \mathds{1}_{((2k+1)Q)^c})(z) \rangle \\
    			& = \biggl\langle v^*, \int_{((2k+1)Q)^c} [K_N(y-w) - K_N(z-w)] f(w) \d w \biggr\rangle \\
    			& \le\int_{\abs{w} \ge (k+ \frac{1}{2}) \ell(Q)} \abs{\langle [K_N^*(w-(z-y)) - K_N^*(w)]v^*, f(z-w) \rangle} \d w \\
    			& \le \sum_{j=0}^{\infty} \biggl( \int_{\abs{w} \in 2^j (k + \frac{1}{2}) \ell(Q) [1,2]} \norm{[K_N^*(w-(z-y)) - K_N^*(w)]v^*}^{q'} \d w \biggr)^{1/q'} \\
    			& \qquad \cdot \biggl( \int_{\abs{w} \in 2^j (k + \frac{1}{2}) \ell(Q) [1,2]} \norm{f(z-w)}^q \d w \biggr)^{1/q}.	
    		\end{split}
		\end{equation}
		We now apply the estimates of Lemma~\ref{lem:kernel_estimate} to the multiplier $m^*$. For this notice that $X^*$ has Fourier type $t$ whenever $X$ has Fourier type $t$. By assumption we can choose some $d \in (\frac{n}{t},l] \cap \IN$. Further, after choosing a slightly smaller Fourier type if necessary, we may also assume that $d \in (\frac{n}{t},\frac{n}{t}+1)$. Then by Lemma~\ref{lem:kernel_estimate} and $q' \le t'$, the first factors in the inner sum satisfy the estimate
		\begin{equation}\label{eq:difference_estimate_2}
    		\begin{split}
    			\MoveEqLeft \biggl( \int_{2^{j+1} (k+ \frac{1}{2}) \ell(Q) \ge \abs{w} \ge 2^j (k+ \frac{1}{2}) \ell(Q)} \norm{[K^*_N(w-(z-y)) - K^*_N(w)]v^*}^{q'} \d w \biggr)^{1/q'} \\
    			& \lesssim (2^j \ell(Q))^{-d} \abs{z-y}^{d-\frac{n}{q}}.
    		\end{split}
		\end{equation}
		From~\eqref{eq:difference_estimate_1}, \eqref{eq:difference_estimate_2} and the dimensional estimate $\abs{z-y} \le \sqrt{n} \ell(Q)$ we obtain by taking in~\eqref{eq:difference_estimate_1} the supremum over all $v^*$ in the norming subset $\mathcal{N}$
		\begin{equation}
			\label{eq:comparability}
    		\begin{split}
    			\MoveEqLeft \norm{T_N(f \mathds{1}_{((2k+1)Q)^c})(y) - T_N(f \mathds{1}_{((2k+1)Q)^c})(z)} \\
    			& \lesssim \sum_{j=0}^{\infty} (2^j \ell(Q))^{-d} \abs{z-y}^{d-\frac{n}{q}} \biggl( \int_{\abs{w} \in 2^j (k + \frac{1}{2}) \ell(Q) [1,2]} \norm{f(z-w)}^q \d w \biggr)^{1/q} \\
    			& \lesssim \sum_{j=0}^{\infty} (2^j \ell(Q))^{-d} \ell(Q)^{d-\frac{n}{q}} (2^{j+1}\ell(Q))^{\frac{n}{q}} \\
    			& \qquad \cdot \biggl( \Bigl(2^{j+1} \Bigl(k+\frac{1}{2} \Bigr) \ell(Q) \Bigr)^{-n} \int_{\abs{z-w} \le 2^{j+1} (k+\frac{1}{2}) \ell(Q)} \norm{f(w)}^q \d w \biggr)^{1/q} \\
    			& \lesssim (M_q \norm{f})(z) \sum_{j=0}^{\infty} 2^{-j(d-\frac{n}{q})}.
    		\end{split}
		\end{equation}
		Here we use the maximal function $(M_q f)(x) \coloneqq (\sup_{Q \ni x} \abs{Q}^{-1} \int_Q \abs{f}^q)^{1/q}$. Notice that the series converges because of $d > \frac{n}{t} > \frac{n}{q}$. Now,~\eqref{eq:comparability} and $x \in Q$ yield
		\begin{align*}
			\MoveEqLeft \sup_{y \in Q} \normalnorm{T_N(f \mathds{1}_{((2k+1)Q)^c})(y)} \lesssim \inf_{z \in Q} \bigl(\normalnorm{T_N(f \mathds{1}_{((2k+1)Q)^c})(z)} + (M_q \norm{f})(z) \bigr) \\
			& \le \inf_{z \in Q} \bigl(\normalnorm{(T_N(f))(z)} + \normalnorm{T_N(f \mathds{1}_{(2k+1)Q})(z)} + (M_q \norm{f})(z) \bigr) \\
			& \le \frac{1}{\abs{Q}} \int_Q \norm{(T_Nf)(z)} \d z + \frac{1}{\abs{Q}} \int_Q \normalnorm{(T_N(f \mathds{1}_{(2k+1)Q}))(z)} \d z + \frac{1}{\abs{Q}} \int_Q (M_q \norm{f})(z) \d z \\
			& \le (M \norm{T_Nf})(x) + (M\normalnorm{T_N f\mathds{1}_{(2k+1)Q}})(x) + M(M_q \norm{f})(x).
		\end{align*}
		Since $Q$ is an arbitrary cube containing $x$, we have the pointwise domination
		\begin{equation*}
			\mathcal{M}_{T_{N,k}}f(x) \lesssim M(\norm{T_Nf})(x) + M(\normalnorm{T_N f \mathds{1}_{(2k+1)Q}})(x) + M(M_q \norm{f})(x).
		\end{equation*}
		Using the fact that $M$ maps $L^{q,\infty}(\IR^n)$ boundedly into itself for all $q > 1$, we obtain for $f \in L^1(\IR^n;X) \cap L^q(\IR^n;X)$
		\begin{align*}
			\MoveEqLeft \normalnorm{\mathcal{M}_{T_{N,k}} f}_{L^{q,\infty}} \lesssim \normalnorm{M \norm{T_N f}}_{L^{q,\infty}} + \normalnorm{M \normalnorm{T_N(f \mathds{1}_{(2k+1)Q})}}_{L^{q,\infty}} + \normalnorm{M(M_q \normalnorm{f})}_{L^{q,\infty}} \\
			& \lesssim \normalnorm{T_N f}_{L^{q,\infty}} + \normalnorm{T_N(f\mathds{1}_{(2k+1)Q})}_{L^{q,\infty}} + \normalnorm{M_q \normalnorm{f}}_{L^{q,\infty}} \lesssim \norm{f}_{L^q(\IR^n;X)}.
		\end{align*}
		For the last term we used the weak type boundedness $M_q\colon L^q \to L^{q,\infty}$, whereas the estimate for the first two is a consequence of our made assumptions.
		
		Now, let $f \in L^q(\IR^n;X)$. Since convergence in $L^{q,\infty}(\IR^n;Y)$ implies pointwise convergence almost everywhere after passing to some subsequence, we can find $f_n \in L^q(\IR^n;X) \cap L^1(\IR^n;X)$ and $(N_n)_{n \in \IN}$ with $f_n \to f$ in $L^q(\IR^n;X)$ and $(T_{N_n} f_n)(x) \to (Tf)(x)$ almost everywhere. The first part of the proof then gives
		\begin{align*}
			(\mathcal{M}_{T,k}f)(x) \le \liminf_{n \to \infty} (\mathcal{M}_{T_{N_n},k}f_n)(x).
		\end{align*}
		The result now follows from Fatou's lemma for weak type $L^q$-spaces.
	\end{proof}

\subsection{Weighted multiplier results}
	
	We now use the weighted estimate in Theorem~\ref{thm:sharp_estimate_sparse} together with the domination established in Lemma~\ref{lem:domination} to obtain some weighted estimates for Fourier multipliers.
	
	\begin{theorem}\label{thm:multiplier_two_weight}
		Let $X,Y$ be Banach spaces with $X$ of non-trivial Fourier type $t > 1$.  Let $m \in L^{\infty}(\IR^{n};\mathcal{B}(X,Y))$ and assume that there exist $l \in (\frac{n}{t},n] \cap \IN$ and $s \in [t,2]$  such that for some $C \ge 0$ one has for all $R > 0$ and $v^* \in \mathcal{N}$ in some norming subset $\mathcal{N} \subset Y^*$
		\begin{equation}\label{eq:assumption_multiplier}
			\biggl( R^{s\abs{\alpha} - n} \int_{R < \abs{\xi} < 2R} \norm{\partial^{\alpha} m^*(\xi)v^*}^s \d\xi \biggr)^{1/s} \le C \norm{v^*} \qquad \text{for all } \abs{\alpha} \le l.
		\end{equation}
		If $T_m\colon L^q(\IR^n;X) \to L^{q,\infty}(\IR^n;Y)$ is bounded for some $q \in [t, \infty)$, then, for all $p, r \in (1, \infty)$ with $p > r \ge q$ and weights $\omega, \sigma\colon \IR^n \to \IR_{\ge 0}$ with $\omega, \sigma^r \in \mathcal{A}_{\infty}(\IR^n; \mathcal{Q}_n)$ and $[\omega,\sigma]_{\mathcal{A}_p^r(\IR^n;\mathcal{Q}_n)} < \infty$, the multiplier operator $T_m(\,\cdot\, \sigma) \colon L^p_{\sigma^r}(\IR^n;X) \to L^p_{\omega}(\IR^n;Y)$ is bounded with
			\[
				\norm{T_m(\sigma f)}_{L^p_{\omega}(\IR^n;Y)} \lesssim [\omega,\sigma]_{\mathcal{A}_p^r(\IR^n;\mathcal{Q}_n)} ([\omega]_{\mathcal{A}_{\infty}(\IR^n;\mathcal{Q}_n)}^{1/p'} + [\sigma^r]_{\mathcal{A}_{\infty}(\IR^n;\mathcal{Q}_n)}^{1/p}) \norm{f}_{L^p_{\sigma^r}(\IR^n;X)}.
			\]
Here the implicit constant only depends on $n$, $p$, $r$, $s$, $l$, $t$, $\norm{\mathcal{F}}_{L^{t}(\IR^{n};X) \to L^{t'}(\IR^{n};X)}$, $C$, $q$ and $\norm{T}_{L^{q} \to L^{q,\infty}}$.
		\end{theorem}
\begin{proof}
Using density of $C_c^{\infty}(\IR^n;X)$ in $L^p_{\sigma^r}(\IR^n;X)$, this follows directly from a combination of Theorem~\ref{thm:sharp_estimate_sparse} and Lemma~\ref{lem:domination}.
\end{proof}
	
\begin{remark}\label{rmk:thm:multiplier_two_weight}
Given $l \in (\frac{n}{t},n] \cap \IN$, the conditions $s \in [t,2]$ and $q \in [t,\infty)$ in Theorem~\ref{thm:multiplier_two_weight} can be relaxed to $s \in (\frac{n}{l},2]$ and $q \in (\frac{n}{l},\infty)$, respectively.
\end{remark}
\begin{proof}
Assume $s \in (\frac{n}{l},2]$ and $q \in (\frac{n}{l},\infty)$ instead of $s \in [t,2]$ and $q \in [t,\infty)$, respectively.
$X$ has Fourier type $\tilde{t}$ for all $\tilde{t} \in [1, t]$ and because of $t, q, s > \frac{n}{l}$ we can choose $\tilde{t} \in (\frac{n}{l},\min (q,s)]$.
Now, applying Theorem~\ref{thm:multiplier_two_weight} with $\tilde{t}$ instead of $t$ and with $\sigma$ and $\omega$ as the Lebesgue measure, we find that $T_m\colon L^{\tilde{q}}(\IR^n;X) \to L^{\tilde{q}}(\IR^n;Y)$ is bounded for all $\tilde{q} \in (q, \infty)$. In particular, we can take $\tilde{q} \in [t,\infty)$.
	\end{proof}

	We now restate the above result in terms of the mapping properties of $T_m$.
	
	\begin{corollary}\label{cor:estimates_mutliplier}
		Let $X,Y$ be Banach spaces with $X$ of non-trivial Fourier type $t > 1$.  Let $m \in L^{\infty}(\IR^{n};\mathcal{B}(X,Y))$ and assume that there exist $l \in (\frac{n}{t},n] \cap \IN$ and $s \in (\frac{n}{l},2]$  such that for some $C \ge 0$ one has for all $R > 0$ and $v^* \in \mathcal{N}$ in some norming subset $\mathcal{N} \subset Y^*$
		\begin{equation*}
			\biggl( R^{s\abs{\alpha} - n} \int_{R < \abs{\xi} < 2R} \norm{\partial^{\alpha} m^*(\xi)v^*}^s \d\xi \biggr)^{1/s} \le C \norm{v^*} \qquad \text{for all } \abs{\alpha} \le l.
		\end{equation*}
		If $T_m\colon L^q(\IR^n;X) \to L^{q,\infty}(\IR^n;Y)$ for some $q \in (\frac{n}{l}, \infty)$, then for all $p, r \in (1, \infty)$ with $p > r \ge q$
			\begin{align}\label{eq:cor:estimates_mutliplier}
				\MoveEqLeft \norm{T_m f}_{L^p_{\omega}(\IR^n;Y)} \lesssim [\omega, \sigma^{-\frac{1}{p-r}}]_{\mathcal{A}_p^r(\IR^n;\mathcal{Q}_n)} ([\omega]_{\mathcal{A}_{\infty}(\IR^n;\mathcal{Q}_n)}^{1/p'} + [\sigma^{-\frac{r}{p-r}}]_{\mathcal{A}_{\infty}(\IR^n;\mathcal{Q}_n)}^{1/p}) \norm{f}_{L^p_{\sigma}(\IR^n;X)},
			\end{align}
where the implicit constant only depends on $n$, $p$, $r$, $s$, $l$, $t$, $\norm{\mathcal{F}}_{L^{t}(\IR^{n};X) \to L^{t'}(\IR^{n};X)}$, $C$, $q$ and $\norm{T}_{L^{q} \to L^{q,\infty}}$.
			In particular, in the one-weight case one has
			\begin{equation*}
				\norm{T_m f}_{L^p_{\omega}(\IR^n;Y)} \lesssim [\omega]_{\mathcal{A}_{p/r}(\IR^n;\mathcal{Q}_n)}^{1/p} ([\omega]_{\mathcal{A}_{\infty}(\IR^n;\mathcal{Q}_n)}^{1/p'} + [\omega^{-\frac{r}{p-r}}]_{\mathcal{A}_{\infty}(\IR^n;\mathcal{Q}_n)}^{1/p}) \norm{f}_{L^p_{\omega}(\IR^n;X)}.
			\end{equation*}
	\end{corollary}
	\begin{proof}
		Ignoring the known dependencies on the weights, Theorem~\ref{thm:multiplier_two_weight} and Remark~\ref{rmk:thm:multiplier_two_weight} give for $f \in L^p_{\sigma_r}(\IR^n;X)$ an estimate of the form
		\begin{equation*}
			\int_{\IR^n} \norm{T_m (f \sigma)}^p \omega \d x \lesssim \int_{\IR^n} \norm{f}^p \sigma^r \d x.
		\end{equation*}
		Now, using the substitution $g = f \sigma$ we get
		\begin{equation*}
			\int_{\IR^n} \norm{T_m (g)}^p \omega \d x \lesssim \int_{\IR^n} \norm{g}^p \sigma^{r-p} \d x.
		\end{equation*}
		Renaming the weights appropriately, we get the two-weight estimate. For the one-weight case we choose $\sigma = \omega^{-\frac{1}{p-r}}$. The two weight characteristic then reduces to
		\begin{equation*}
			[\omega,\sigma] = \sup_{Q} \biggl( \frac{1}{\abs{Q}} \int \omega^{-\frac{r}{p-r}} \biggr)^{\frac{1}{r}-\frac{1}{p}} \biggl( \frac{1}{\abs{Q}} \int_Q \omega \biggr)^{\frac{1}{p}} = [\omega]_{\mathcal{A}_{p/r}}^{1/p}. \qedhere
		\end{equation*}
	\end{proof}

A duality argument now gives the following result, where we replace the weak-$L^q$-bound by an $L^q$-bound. Note that in all results valid choices of the constants give $\frac{n}{l} \in [1,2)$.

\begin{corollary}\label{cor:estimates_mutliplier_dual}
		Let $X,Y$ be Banach spaces with $Y$ of non-trivial Fourier type $t > 1$.  Let $m \in L^{\infty}(\IR^{n};\mathcal{B}(X,Y))$ and assume that there exist $l \in (\frac{n}{t},n] \cap \IN$ and $s \in (\frac{n}{l},2]$ such that for some $C \ge 0$ one has, for all $R > 0$ and $x \in X$,
		\begin{equation*}
			\biggl( R^{s\abs{\alpha} - n} \int_{R < \abs{\xi} < 2R} \norm{\partial^{\alpha} m(\xi)x}^s \d\xi \biggr)^{1/s} \le C\norm{x} \qquad \text{for all } \abs{\alpha} \le l.
		\end{equation*}
		If $T_m\colon L^q(\IR^n;X) \to L^{q}(\IR^n;Y)$ for some $q \in (1, (\frac{n}{l})')$, then, for all $p, r \in (1, \infty)$ with $p < r' < q'$,
			\begin{align}
				\MoveEqLeft \norm{T_m f}_{L^p_{\omega}(\IR^n;Y)} \lesssim [\sigma^{-\frac{1}{p-1}}, \omega^{-\frac{1}{(p-1)(p'-r)}}]_{\mathcal{A}_{p'}^r(\IR^n;\mathcal{Q}_n)} \nonumber \\
				& \qquad ([\sigma^{-\frac{1}{p-1}}]_{\mathcal{A}_{\infty}(\IR^n;\mathcal{Q}_n)}^{1/p} + [\omega^{\frac{r}{(p-1)(p'-r)}}]_{\mathcal{A}_{\infty}(\IR^n;\mathcal{Q}_n)}^{1/p'}) \norm{f}_{L^p_{\sigma}(\IR^n;X)}, \label{eq:cor:estimates_mutliplier_dual}
			\end{align}
where the implicit constant only depends on $n$, $p$, $r$, $s$, $l$, $t$, $\norm{\mathcal{F}}_{L^{t}(\IR^{n};Y) \to L^{t'}(\IR^{n};Y)}$, $C$, $q$ and $\norm{T}_{L^{q} \to L^{q}}$.
			In particular, in the one weight case one has
			\begin{align*}
				\MoveEqLeft \norm{T_m f}_{L^p_{\omega}(\IR^n;Y)}\lesssim [\omega^{-\frac{1}{p-1}}]_{\mathcal{A}_{p'/r}(\IR^n;\mathcal{Q}_n)}^{1/p'} ([\omega^{-\frac{1}{p-1}}]_{\mathcal{A}_{\infty}(\IR^n;\mathcal{Q}_n)}^{1/p} + [\omega^{\frac{r}{(p-1)(p'-r)}}]_{\mathcal{A}_{\infty}(\IR^n;\mathcal{Q}_n)}^{1/p'}) \norm{f}_{L^p_{\omega}(\IR^n;X)}.
			\end{align*}
	\end{corollary}
	\begin{proof}
As $m^{**}(\xi)x = m(\xi)x$ for $x \in X$, $X \subset X^{**}$ is norming for $X^{*}$ and $Y^{*}$ has Fourier type $t$, duality for Fourier multipliers (see Section~\ref{sec:r-boundedness}) gives that $\tilde{m}^{*}$ satisfies the assumptions of Corollary~\ref{cor:estimates_mutliplier} with $q'$ instead of $q$. Dualizing the corresponding estimate yields the desired result.
	\end{proof}
	
	If we replace~\eqref{eq:assumption_multiplier} by its stronger operator norm variant and the weak-$L^q$ bound by an $L^q$-bound, we obtain the following result valid for a broader range of indices.
	
	\begin{corollary}\label{cor:estimates_mutliplier_combi}
		Let $X,Y$ be Banach spaces of non-trivial Fourier type $t > 1$. Let $m \in L^{\infty}(\IR^{n};\mathcal{B}(X,Y))$ and assume that there exist $l \in (\frac{n}{t},n] \cap \IN$ and $s \in (\frac{n}{l},2]$ such that for some $C \ge 0$ one has for all $R > 0$
		\begin{equation}\label{eq:cor:estimates_mutliplier_combi}
			\biggl( R^{s\abs{\alpha} - n} \int_{R < \abs{\xi} < 2R} \norm{\partial^{\alpha} m(\xi)}^s \d\xi \biggr)^{1/s} \le C \qquad \text{for all } \abs{\alpha} \le l.
		\end{equation}
 Suppose that $T_m\colon L^q(\IR^n;X) \to L^{q}(\IR^n;Y)$ for some $q \in (\frac{n}{l}, (\frac{n}{l})')$. Then \eqref{eq:cor:estimates_mutliplier} holds for all $p, r \in (1, \infty)$ with $p > r > \frac{n}{l}$ and \eqref{eq:cor:estimates_mutliplier_dual} holds for all $p, r \in (1, \infty)$ with $p < r' < (\frac{n}{l})'$.
	\end{corollary}
	\begin{proof}
Note that the conditions of Corollaries \ref{cor:estimates_mutliplier} and \ref{cor:estimates_mutliplier_dual} are both satisfied.
Applying them with $\omega=\sigma=\mathds{1}$,
we find that $T_{m}$ is $L^{\tilde{q}}$-bounded for all $\tilde{q} \in (q,\infty)$ and for all $\tilde{q} \in (1,q')$, respectively. Interpolation subsequently yields $L^{\tilde{q}}$-boundedness for all $\tilde{q} \in (1,\infty)$.
In particular, we can apply Corollary~\ref{cor:estimates_mutliplier} for all $q \in (\frac{n}{l}, \infty)$ and we can apply Corollary~\ref{cor:estimates_mutliplier_dual} for all $q \in (1,(\frac{n}{l})')$.
	\end{proof}

Combining the above result with
\cite[Corollary~4.4]{GirWei03} we obtain the following corollary (Corollary~\ref{cor:estimates_mutliplier;GirWei03}).

For the statement of this corollary it will be convenient to introduce the following notation. Let $N \in \IN$ and $q \in \{1,\infty\}$.
We denote by $\mathcal{R}\mathfrak{M}^{n}_{N,q}(X,Y)$ the space of all symbols $m \in L_{\infty}(\IR^{n};\mathcal{B}(X,Y))$ with $D^{\alpha}m_{|\IR^{n} \setminus \{0\}} \in L^{1}_{\mathrm{loc}}(\IR^{n} \setminus \{0\};\mathcal{B}(X,Y))$ for each $|\alpha|_{q} \leq N$ such that
\begin{equation}\label{eq:R-Miklhin-norm}
\norm{m}_{\mathfrak{M}^{n}_{N,q}} := \sup_{|\alpha|_{q} \leq N} \mathcal{R}_{\mathrm{ess}}\left\{|\xi|^{|\alpha|}D^{\alpha}m(\xi) : \xi \neq 0 \right\} < \infty,
\end{equation}
where $\mathcal{R}_{\mathrm{ess}}$ is the essential $\mathcal{R}$-bound; given $f \in L^{0}(\IR^{n} \setminus \{0\};\mathcal{B}(X,Y))$, $\mathcal{R}_{\mathrm{ess}}\{f(\xi): \xi \neq 0\}$ is the infinimum over all representatives $g$ of the equivalence class of $f$ (a.e.\ coincidence) of $\mathcal{R}\{g(\xi): \xi \neq 0\}$.

\begin{corollary}\label{cor:estimates_mutliplier;GirWei03}
Let $X,Y$ be Banach spaces of non-trivial Fourier type $t > 1$ and let $l \in (\frac{n}{t},n] \cap \IN$.
Then
			\begin{align*}
				\MoveEqLeft \norm{T_m}_{L^p_{\sigma} \to L^p_{\omega}} \lesssim_{X,Y,t,l,n,p,r} [\omega, \sigma^{-\frac{1}{p-r}}]_{\mathcal{A}_p^r(\IR^n;\mathcal{Q}_n)} \\ & \qquad\qquad ([\omega]_{\mathcal{A}_{\infty}(\IR^n;\mathcal{Q}_n)}^{1/p'} + [\sigma^{-\frac{r}{p-r}}]_{\mathcal{A}_{\infty}(\IR^n;\mathcal{Q}_n)}^{1/p}) \norm{m}_{\mathcal{R}\mathfrak{M}^{n}_{l,1}(X,Y)}
			\end{align*}
for $p, r \in (1, \infty)$ with $p > r > \frac{n}{l}$ and
			\begin{align*}
				\MoveEqLeft \norm{T_m}_{L^p_{\sigma} \to L^p_{\omega}} \lesssim_{X,Y,t,l,n,p,r} [\sigma^{-\frac{1}{p-1}}, \omega^{-\frac{1}{(p-1)(p'-r)}}]_{\mathcal{A}_{p'}^r(\IR^n;\mathcal{Q}_n)} \nonumber \\
				& \qquad\qquad ([\sigma^{-\frac{1}{p-1}}]_{\mathcal{A}_{\infty}(\IR^n;\mathcal{Q}_n)}^{1/p} + [\omega^{\frac{r}{(p-1)(p'-r)}}]_{\mathcal{A}_{\infty}(\IR^n;\mathcal{Q}_n)}^{1/p'}) \norm{m}_{\mathcal{R}\mathfrak{M}^{n}_{l,1}(X,Y)}.
\end{align*}
for $p, r \in (1, \infty)$ with $p < r' < (\frac{n}{l})'$.
\end{corollary}	
\begin{proof}
By \cite[Corollary~4.4]{GirWei03}, $T_{m}$ is $L^{q}$-bounded for all $q \in (1,\infty)$, thus in particular for some $q \in (\frac{n}{l}, (\frac{n}{l})')$.
As the $\mathcal{R}$-boundedness condition in the definition of $\mathcal{R}\mathfrak{M}^{n}_{l,1}(X,Y)$ implies the integral condition \eqref{eq:cor:estimates_mutliplier_combi}, we may apply Corollary~\ref{cor:estimates_mutliplier_combi} to obtain the desired result.
\end{proof}
	
	\begin{remark}\label{rmk:one-weight}
		Replacing the $\mathcal{A}_{\infty}$-constants by the larger $\mathcal{A}_{p/r}$-constants in the one weight setting, i.e.\ using the estimates $[\omega]_{\mathcal{A}_{\infty}} \lesssim [\omega]_{\mathcal{A}_{p/r}}$ and $[\omega^{-r/(p-r)}]_{\mathcal{A}_{\infty}} \lesssim [\omega]_{\mathcal{A}_{p/r}}^{r/(p-r)}$ respectively, we have, a fortiori, under the assumptions of Corollary~\ref{cor:estimates_mutliplier} and~\ref{cor:estimates_mutliplier_dual} (and the subsequent results), for $T_m\colon L^p_{\omega}(\IR^n;X) \to L^p_{\omega}(\IR^n;Y)$ the estimates
		\begin{equation*}
			\norm{T_m} \lesssim [\omega]_{\mathcal{A}_{p/r}}^{\max(1, \frac{1}{p-r})} \quad \text{and} \quad \norm{T_m} \lesssim [\omega^{-\frac{1}{p-1}}]_{\mathcal{A}_{p'/r}}^{\max(1, \frac{1}{p'-r})},
		\end{equation*}
		respectively. In particular, $T_m$ is bounded for $\omega \in \mathcal{A}_{p/r}$ and $\omega^{-\frac{1}{p-1}} \in \mathcal{A}_{p'/r}$ respectively.
	\end{remark}
	
	\begin{remark}
		The boundedness result without the dependencies on the weights stated in the previous remark follows
from earlier known results: since the kernel $K$ satisfies the pointwise $p$-Hörmander condition, one can essentially apply a variant of~\cite[Part~I, Theorem~1.6]{FraRuiTor86} to obtain the boundedness of $T_m$ for the same class of $\mathcal{A}_p$-weights (for this see also the remarks at the end of~\cite{FraRuiTor86}).
	\end{remark}
	
\begin{remark}\label{rmk:cor:estimates_mutliplier;GirWei03}
In connection with Remark~\ref{rmk:one-weight}, note that the fact $\mathcal{A}_{p} = \bigcup_{q \in (1,p)}\mathcal{A}_{q}$ yields that for each $\omega \in \mathcal{A}_{p}$ there exists $r \in (1,p)$ such that $\omega \in \mathcal{A}_{p/r}$.

In particular, Corollory~\ref{cor:estimates_mutliplier;GirWei03} gives a Mikhlin theorem for $\mathcal{A}_{p}$-weights, but with implied constants that have a complicated dependence on the weight.
A nice dependence on the weight can be obtained at the cost of increasing the Mikhlin condition to order $n+2$ (where the higher order estimates only require uniform boundedness instead of $R$-boundedness), see \cite[Proposition~3.1]{MeyVer15}.
This smoothness $n+2$ could actually be improved to $n+1$ by using (something in the spirit of) \cite{Hyt04} instead of \cite[Proposition~VI.4.4.2(a)]{Ste93} for passing from the Fourier multiplier perspective to the perspective of singular integrals.
\end{remark}

	\begin{remark}
		In the case of scalar multipliers $m\colon \IR^n \to \IC$ the assumptions made on the multiplier always imply the classical Hörmander condition. Further, the boundedness of $T_m\colon L^2(\IR^n) \to L^2(\IR^n)$ is equivalent to $m \in L^{\infty}(\IR^n)$ by Plancherel's theorem. Hence, for scalar multipliers we recover the one-weight results in~\cite[Theorem~1]{KurWhe79} whose proof uses properties of the sharp maximal function instead of domination by sparse operators.
	\end{remark}

\bibliographystyle{plain}

\end{document}